\documentclass[12pt,reqno]{amsproc}%
\usepackage{amsfonts}
\usepackage{amsmath}
\usepackage{amssymb}
\usepackage{graphicx}%
 \theoremstyle{plain}
\newtheorem{theorem}{Theorem}[section]

\newtheorem{corollary}[theorem]{Corollary}

\newtheorem{definition}[theorem]{Definition}
\newtheorem{example}[theorem]{Example}

\newtheorem{lemma}[theorem]{Lemma}

\newtheorem{proposition}[theorem]{Proposition}
\newtheorem{remark}[theorem]{Remark}

\numberwithin{equation}  {section}

\begin{document}

\title{Hochschild cohomology of type II$_1$ von Neumann algebras with Property $\Gamma$}

\author{Wenhua Qian}
\address{Wenhua Qian \\
        Demartment of Mathematics \\
         East China University of Science and Technology\\
         Shanghai 200237, China;   Email: whqian86@163.com}
\author{Junhao Shen}
\address{Junhao Shen \\
        Demartment of Mathematics and Statistics \\
         University of New Hampshire\\
         Durham, NH 03824;   Email:  Junhao.Shen@unh.edu}

\begin{abstract}
 In this paper, Property $\Gamma$ for a type II$_{1}$ von Neumann
algebra is introduced as a generalization of Murray and von
Neumann's Property $\Gamma$ for a type II$_{1}$ factor. The main
result of this paper is that if a type II$_{1}$ von Neumann algebra
$\mathcal{M}$ with separable predual has Property $\Gamma$, then the
continuous Hochschild cohomology group $H^{k}(\mathcal{M},
\mathcal{M})$ vanishes for every $k \geq 2$. This gives a
generalization of an earlier result in \cite{EFAR2}.

\end{abstract}


\maketitle

\section{Introduction}
The continuous Hochschild cohomology of von Neumann algebras was
initialized by Johnson, Kadison and Ringrose in \cite{RJ1},
\cite{RJ2}, \cite{BRJ}, where it was conjectured that the $k$-th
continuous Hochschild cohomology group $H^{k}(\mathcal{M},
\mathcal{M})$ is trivial for any von Neumann algebra $\mathcal{M}$,
$k \geq 1$. In the case $k=1$, this conjecture, which is equivalent
to the problem of whether a derivation of a von Neumann algebra into
itself is inner, had been solved by Kadison and Sakai independently
in \cite{Ka1}, \cite{S1}. In the following we focus on the case when
$k \geq 2$. In \cite{BRJ}, it was shown   that $H^{k}(\mathcal{M},
\mathcal{M})=0$ for $k\ge 2$ if $\mathcal{M}$ is a injective von
Neumann algebra. It follows that if $\mathcal M$ is a type I von
Neumann algebra, then $H^{k}(\mathcal{M}, \mathcal{M})=0$ for $k\ge
2$.

Significant progress was made after the introduction of
completely bounded Hochschild cohomology groups for von Neumann
algebras (\cite{EGA}, \cite{EFAR1}, \cite{EFAR2}, \cite{EA0},
\cite{EA1}, \cite{EA2}, \cite{FR},    \cite{PS}, \cite{AR1},
\cite{AR2}). It was shown in \cite{EA1}, \cite{CS} (see also
\cite{AR3}) that the completely bounded Hochschild cohomology group
$H^{k}_{cb}(\mathcal{M}, \mathcal{M})=0$ for $k \geq 2$. As a consequence of
  results in \cite{EGA}, if $\mathcal M$ is a type II$_\infty$ or type III von
Neumann algebra, then $H^{k}(\mathcal{M}, \mathcal{M})=0$ for $k\ge
2$.  In the case that $\mathcal M$ is a type II$_1$ von Neumann
algebra, many results as listed below have also been obtained.  (We
refer to a wonderful book \cite{AR3} by Sinclair and Smith for a
survey of Hochschild cohomology theory for von Neumann algebras and
proofs of most of the following results.)
\begin{enumerate}
\item [(i)] $H^{k}(\mathcal{M}, \mathcal{M})=0$   for $k \geq 2$ if the
type II$_1$ central summand in the type decomposition
$\mathcal{M}=\mathcal{M}_{1}\oplus \mathcal{M}_{c_{1}} \oplus
\mathcal{M}_{c_{\infty}} \oplus \mathcal{M}_{\infty}$ of the von
Neumann algebra $\mathcal{M}$ satisfies $\mathcal{M}_{c_{1}} \otimes
\mathcal{R} \cong \mathcal{M}_{c_{1}}$, where $\mathcal{R}$ is the
hyperfinite type II$_{1}$ factor  (\cite{EGA}).

\item [(ii)] $H^{k}(\mathcal{M}, \mathcal{M})=0$  for $k \geq 2$  if
$\mathcal{M}$ is a type  II$_{1}$ von Neumann algebra with a Cartan
subalgebra and separable predual (\cite{EFAR1}, \cite{FR}, \cite{AR1}, \cite{AR2}); it was shown later in \cite{Ca} that $H^{k}(\mathcal{M}, \mathcal{M})=0$ if $\mathcal{M}$ is a type II$_{1}$ factor with a Cartan masa.

\item [(iii)] $H^{2}(\mathcal{M}, \mathcal{M})=0$   for $k \geq 2$ if $\mathcal{M}$ is a
type  II$_{1}$ factor satisfying various technical properties
related to its action on $L^{2}(\mathcal{M}, tr)$  (\cite{FR}).

\item [(iv)] $H^{k}(\mathcal{M}, \mathcal{M})=0$  for $k \geq 2$  if
$\mathcal{M}$ is a type  II$_{1}$ factor with Property $\Gamma$
(\cite{EFAR2}).

\item [(v)] $H^{2}(\mathcal{M}_{1}  {\otimes} \mathcal{M}_{2},
\mathcal{M}_{1}  {\otimes} \mathcal{M}_{2})=0$  if both
$\mathcal{M}_{1}$ and $\mathcal{M}_{2}$ are type  II$_{1}$ von
Neumann algebras (\cite{PS}).
\end{enumerate}

A motivation of this paper is to generalize the listed result (iv) in
\cite{EFAR2} for type II$_{1}$ factors with Property $\Gamma$ to
general type II$_{1}$ von Neumann algebras with certain properties.
Recall Murray and von Neumann's Property $\Gamma$ for type II$_{1}$
factors as follows. {\em Suppose $\mathcal{A}$ is a type II$_{1}$
factor with a trace $\tau$. Let $\Vert \cdot \Vert_{2}$ be the
$2$-norm on $\mathcal{A}$ given by $\Vert a
\Vert_{2}=\sqrt{\tau(a^{*}a)}$ for any $a \in \mathcal{A}$. Then
$\mathcal{A}$ has Property $\Gamma$ if, given $\epsilon >0$ and
$a_{1}, a_{2}, \dots, a_{k} \in \mathcal{A}$, there exists a unitary
$u \in \mathcal{A}$ such that

\begin{enumerate}
\item [(a)] $\tau(u)=0$;

\item [(b)] $\Vert ua_{j}-a_{j}u \Vert_{2} < \epsilon, \ \  \forall \  1 \leq j
\leq k$.
\end{enumerate}}
An equivalent definition of Property $\Gamma$ for a type II$_{1}$
factor $\mathcal{A}$ was given by Dixmier in
\cite{Di1}. {\em Suppose $\mathcal{A}$ is a type II$_{1}$ factor with a trace $\tau$. Let
$\Vert \cdot \Vert_{2}$ be the $2$-norm on $\mathcal{A}$ given by
$\Vert a \Vert_{2}=\sqrt{\tau(a^{*}a)}$ for any $a \in
\mathcal{A}$. Then $\mathcal{A}$ has Property $\Gamma$ if, given $n
\in \mathbb{N}$, $\epsilon >0$ and $a_{1}, a_{2}, \dots, a_{k} \in
\mathcal{A}$, there exist $n$ orthogonal equivalent projections $\{
p_{1}, p_{2}, \dots, p_{n} \}$ in $\mathcal{A}$ with sum $I$ such
that
$$\Vert p_{i}a_{j}-a_{j}p_{i} \Vert_{2} < \epsilon, \qquad  \forall \ 1
\leq i \leq n, 1 \leq j \leq k.$$}

In the paper, we extend Dixmier's equivalent definition of Murray
and von Neumann's Property $\Gamma$ to general von Neumann algebras
as follows.

\vspace{0.5cm}

   { Definition \ref{3.1}.} \   {\em  Suppose $\mathcal{M}$ is a type II$_{1}$ von Neumann algebra with a
predual $\mathcal M_{\sharp}$. Suppose that $\sigma (\mathcal M,
\mathcal M_\sharp)$ is the weak-$*$ topology on $\mathcal M$ induced
from $\mathcal M_\sharp$. We say that $\mathcal{M}$ has Property
$\Gamma$ if and only if $ \forall \ a_{1}, a_{2}, \dots, a_{k} \in
\mathcal{M}$ and $\forall \ n\in \Bbb N$, there exist a partially ordered set $\Lambda$ and a family of
projections $$\{ p_{i \lambda}: 1\le i\le n; \lambda \in \Lambda \}\subseteq
\mathcal{M}$$ satisfying
\begin{enumerate}
\item [(i)] For each $\lambda \in \Lambda$,   $\{ p_{1 \lambda}, p_{2 \lambda}, \ldots,
p_{n \lambda} \}$ is a family of orthogonal equivalent projections in
$\mathcal M$ with sum $I$.
\item [(ii)] For each $1\le i\le n$ and $1\le j\le k$,
$$
\lim_{\lambda} (p_{i \lambda}a_{j}-a_{j}p_{i
\lambda})^*(p_{i \lambda}a_{j}-a_{j}p_{i \lambda}) =0\qquad \text {in $\sigma(\mathcal
M, \mathcal M_\sharp)$ topology.}$$
\end{enumerate} }
\vspace{0.5cm}

We note that Definition \ref{3.1} coincides with Dixmier's
definition (and  Murray and von Neumann's definition)  when
$\mathcal{M}$ is a type II$_{1}$ factor (see Corollary \ref{3.3.5}). The
following theorem is our main result of this paper, which gives a
generalization of an earlier result in \cite{EFAR2}.

\vspace{0.5cm}

  {  Theorem  \ref{mainthm}}. \   {\em Suppose $\mathcal{M}$ is a type
II$_{1}$ von Neumann algebra with separable predual. If
$\mathcal{M}$ has Property $\Gamma$, then  the Hochschild cohomology
group
$$H^{k}(\mathcal{M}, \mathcal{M})=0, \qquad \forall \ k\ge 2.$$}

 The proof of Theorem \ref{mainthm} follows the similar line as the
 one in \cite{EFAR2} besides that new tools from direct integral
 theory for von Neumann algebras need to be developed.

 The organization of this paper is as follows. In section 3, we introduce a definition of Property $\Gamma$ for type II$_1$ von Neumann
 algebras.
 In
section 4, by applying the technique of direct integrals to
$\mathcal{M}$, we will construct a hyperfinite subfactor
$\mathcal{R}$ such that the relative commutant of $\mathcal{R}$ is
the center of $\mathcal{M}$ and $\mathcal{R}$ satisfies the
additional property of containing an asymptotically commuting family
of projections for $\mathcal{M}$. In section 5, we will prove a
Grothendick inequality for $\mathcal{R}$-multimodular normal
multilinear maps. Then, in section 6, we combine these results
obtained in section 4 and section 5 to show that for a type II$_{1}$
von Neumann algebra $\mathcal{M}$ with separable predual, if
$\mathcal{M}$ has Property $\Gamma$, then every bounded $k$-linear
$\mathcal{R}$-multimodular separately normal map from
$\mathcal{M}^{k}$ to $\mathcal{M}$ is completely bounded, which
implies the triviality of the cohomology group $H^{k}(\mathcal{M}, \mathcal{M})$ by Theorem 3.1.1 and
Theorem 4.3.1 in \cite{AR3}.

\section{Preliminaries}
\subsection{Hochschild cohomology}
In this subsection, we will recall a definition of continuous
Hochschild cohomology groups (see \cite{AR3}).

Let $\mathcal{M}$ be a von Neumann algebra. We say that a Banach space $\mathcal{X}$ is a Banach $\mathcal{M}$-bimodule if there is a module action of $\mathcal{M}$ on both the left and right of $\mathcal{X}$ satisfying

$$\Vert m \xi \Vert \leq \Vert m \Vert \Vert \xi \Vert$$
and
$$\Vert \xi m \Vert \leq \Vert \xi \Vert \Vert m \Vert$$
for any $m \in \mathcal{M}, \xi \in \mathcal{X}$.

For each integer $k \geq 1$, we denote by $\mathcal{L}^{k}(\mathcal{M},\mathcal{X})$ the Banach space of $k$-linear bounded maps $\phi : \mathcal{M}^{k} \to \mathcal{X}$. For $k=0$, we define $\mathcal{L}^{0}(\mathcal{M}, \mathcal{X})$ to be $\mathcal{X}$. Then we can define coboundary operators $\partial^{k}:\mathcal{L}^{k}(\mathcal{M}, \mathcal{X}) \to \mathcal{L}^{k+1}(\mathcal{M}, \mathcal{X}) $ as follows:
\begin{enumerate}
\item [(i)] when $k \geq 1$,  for any $\phi \in \mathcal{L}^{k}(\mathcal{M}, \mathcal{X}), a_{1}, a_{2}, \dots, a_{k} \in \mathcal{M}$,
$$
\begin{aligned}
\partial^{k} \phi(a_{1}, a_{2}, \dots, a_{k+1})&=a_{1}\phi (a_{2}, \dots, a_{k+1})\\
&+\sum_{i=1}^{k}(-1)^{k}\phi (a_{1}, \dots, a_{i-1}, a_{i}a_{i+1}, \dots, a_{k+1})\\
&+(-1)^{k+1}\phi (a_1, \dots, a_{k})a_{k+1}.
\end{aligned}$$

\item [(ii)] when $k=0$, for any $\xi \in \mathcal{X}, a \in \mathcal{M}$, $\partial^{0}\xi (a)=a \xi -\xi a$.
\end{enumerate}
It's easy to check that $\partial^{k} \partial^{k-1}=0$ for each $k \geq 1$. Thus $Im \partial^{k-1}$(the space of coboundaries) is contained in $Ker \partial^{k}$(the space of cocycles). The continuous Hochschild cohomology groups $H^{k}(\mathcal{M}, \mathcal{X})$ are then defined to be the quotient vector spaces $Ker \partial^{k} / Im \partial^{k-1}, k \geq 1$.

\subsection{Direct integral}
The concepts of direct integrals of separable Hilbert spaces and von
Neumann algebras acting on separable Hilbert spaces were introduced
by von Neumann in \cite{vN}. General knowledge on direct integrals
can be found in \cite{vN}, \cite{KR1}. Here, we list some lemmas
which will be needed in this paper.

\begin{lemma} \label{2.1}
(\cite{KR1}) Suppose $\mathcal{M}$ is a von Neumann algebra acting on
a separable Hilbert space $H$. Let $\mathcal{Z}$ be the center of
$\mathcal{M}$. Then there is a direct integral decomposition of
$\mathcal{M}$ relative to $\mathcal{Z}$, i.e. there exists a locally
compact complete separable metric measure space $(X, \mu)$ such that
\begin{enumerate}
\item [(i)] $H$ is (unitarily equivalent to) the direct integral of $\{ H_{s} : s \in X \}$ over $(X, \mu)$, where each $H_{s}$ is a separable Hilbert space, $s \in X$.
\item [(ii)] $\mathcal{M}$ is (unitarily equivalent to) the direct integral
of $\{ \mathcal{M}_{s} \}$ over $(X, \mu)$, where $\mathcal{M}_{s}$
is a factor in $B(H_{s})$ almost everywhere. Also, if $\mathcal{M}$
is of type $I_{n}$($n$ could be infinite), II$_{1}$, II$_{\infty}$
or $III$, then the components $\mathcal{M}_{s}$ are, almost
everywhere, of type  I$_{n}$, II$_{1}$, II$_{\infty}$ or $III$,
respectively.
\end{enumerate}
Moreover, the center $\mathcal{Z}$ is (unitarily equivalent to) the algebra of diagonalizable operators relative to this decomposition.
\end{lemma}

The following lemma gives a decomposition of a normal sate on a direct integral of von Neumann algebras.
\begin{lemma} \label{2.2}
(\cite{KR1}) If $H$ is the direct integral of separable Hilbert
spaces $\{ H_{s} \}$ over $(X, \mu)$, $\mathcal{M}$ is a
decomposable von Neumann algebra on $H$ (i.e every operator in $\mathcal{M}$ is decomposable relative to the direct integral decomposition, see Definition 14.1.6 in \cite{KR1}) and $\rho$ is a normal
state on $\mathcal{M}$. There is a positive normal linear functional
$\rho_{s}$ on $\mathcal{M}_{s}$ for every $s \in X$ such that
$\rho(a)=\int_{X} \rho_{s}(a(s))d\mu$ for each $a$ in
$\mathcal{M}$. If $\mathcal{M}$ contains the algebra $\mathcal{C}$
of diagonalizable operators and $\rho \vert_{E\mathcal{M}E}$ is
faithful or tracial, for some projection $E$ in $\mathcal{M}$, then
$\rho_{s} \vert_{E(s)\mathcal{M}_{s}E(s)}$ is, accordingly,
faithful or tracial almost everywhere.
\end{lemma}

\begin{remark} \label{2.3}
From the proof of Lemma 14.1.19 in \cite{KR1}, we obtain that if
$\rho=\sum\limits_{n=1}^{\infty}\omega_{y_{n}}$ on $\mathcal{M}$,
where $\{ y_{n} \}$ is a sequence of vectors in $H$ such that
$\sum\limits_{n=1}^{\infty}\Vert y_{n} \Vert ^{2}=1$ and
$\omega_{y}$ is defined on $\mathcal{M}$ such that
$\omega_{y}(a)=\langle ay, y \rangle$ for any $a \in \mathcal{M}, y
\in H$, then $\rho_{s}$ can be chosen to be
$\sum\limits_{n=1}^{\infty} \omega_{y_{n}(s)}$ for each $s \in X$.
\end{remark}

\begin{remark} \label{2.4}
Let $\mathcal{M}=\int_{X} \bigoplus M_{s} d \mu$ and $H= \int_{X}
\bigoplus H_{s} d \mu$ be the direct integral decompositions of
$(\mathcal{M},H)$ relative to the center $\mathcal{Z}$ of
$\mathcal{M}$. By the argument in section 14.1 in \cite{KR1}, we can
find a separable Hilbert space $K$ and a family of unitaries
$\{U_{s}:H_{s} \to K; s \in X \}$ such that $s \to U_{s}x(s)$ is
measurable (i.e. $s \to \langle U_{s}x(s), y \rangle$ is measurable
for any vector $y$ in $K$) for every $x\in H$ and $s \to
U_{s}a(s)U_{s}^{*}$ is measurable (i.e. $s \to \langle
U_{s}a(s)U_{s}^{* }y, z \rangle$ is measurable for any vectors
$y,z$ in $K$) for every decomposable operator $a \in B(H)$.
\end{remark}

\begin{proposition} \label{2.5}
 Let $\mathcal{M}$ be a type II$_{1}$ von Neumann algebra acting on a separable Hilbert space H. Let $\mathcal{M}=\int_{X} \bigoplus M_{s} d \mu$ and $H= \int_{X} \bigoplus H_{s} d \mu$ be the direct integral decompositions of $\mathcal{M}$ and $H$ relative to the center $\mathcal{Z}$ of $\mathcal{M}$. Suppose $K$ is a Hilbert space and $\{ U_{s}: H_{s} \to K \}$ is a family of unitaries as in Remark \ref{2.4}. Denote by $\mathcal{B}$ the unit ball of $B(K)$ equipped with the $*$-strong operator topology. Suppose $\rho$ is a faithful normal tracial state on $\mathcal{M}$. Then there is a family of positive, faithful, normal, tracial linear functionals $\rho_s$ on $\mathcal{M}_{s}$ (almost everywhere) such that
\begin{enumerate}
\item [(a)] $\rho (a)=\int_{X}\rho _{s}(a(s))d\mu$ for every $a \in \mathcal{M}$;

\item [(b)] for any $a_0\in\mathcal M$, there exists a Borel $\mu$-null subset N of X such that the maps
$$(s,b) \to \rho_{s}((a_{0}(s)U_{s}^{*}bU_{s}-U_{s}^{*}bU_{s}a_{0}(s))^{*}(a_{0}(s)U_{s}^{*}bU_{s}-U_{s}^{*}bU_{s}a_{0}(s)))$$
and
$$(s,b) \to \rho_{s}(U_{s}^{*}bU_{s})$$
are Borel measurable from $(X \setminus N) \times \mathcal{B}$ to $\mathbb{C}$.\end{enumerate}
\end{proposition}

\begin{proof}

 If $\rho$ is a faithful, normal, tracial state on $\mathcal{M}$, then there exist a sequence of vectors $\{y_{n}\}\subset H$ with $\sum\limits_{n=1}^{\infty} \Vert y_{n} \Vert^{2}=1$ such that $\rho =\sum\limits_{n=1}^{\infty} \omega _{y_{n}}$.  Take $\rho_{s}=\sum\limits_{n=1}^{\infty} \omega _{y_{n}(s)} $ for every $s \in X$. By Remark \ref{2.3}, we know,  for $s\in X$ almost everywhere,   $\rho_s$ is a positive, faithful, normal, tracial linear functional on $\mathcal M_s$ and
 \begin{equation}\rho (a)=\int_{X}\rho _{s}(a(s))d\mu\qquad \forall a \in \mathcal{M}.\label{equa 2.5.1}\end{equation}

For each vector $y_n$ in $ H$, we let $\omega_{{y_n}}$ be the vector state relative to ${y_n}$. Then $$\omega _{{y_n}}(a)=\int_{X}\omega _{{y_n}(s)}(a(s))d\mu \qquad \forall a\in \mathcal{M}.$$
Consider the maps $\phi_n, \psi_n: X \times \mathcal{B} \to \mathbb{C}$:
\begin{eqnarray*}
\phi_n (s,b) = \omega_{{y_n}(s)}((a_{0}(s)U_{s}^{*}bU_{s}-U_{s}^{*}bU_{s}a_{0}(s))^{*}(a_{0}(s)U_{s}^{*}bU_{s}-U_{s}^{*}bU_{s}a_{0}(s)))
\end{eqnarray*}
and
\begin{eqnarray*}
\psi_n (s,b) = \omega_{{y_n}(s)}(U_{s}^{*}bU_{s}).
\end{eqnarray*}
We have
$$\begin{aligned}
 \phi_n (s,b) &= \omega_{{y_n}(s)} ((a_{0}(s)U_{s}^{*}bU_{s}-U_{s}^{*}bU_{s}a_{0}(s))^{*}(a_{0}(s)U_{s}^{*}bU_{s}-U_{s}^{*}bU_{s}a_{0}(s))) \\
& =  \langle (a_{0}{s}U_{s}^{* }bU_{s}-U_{s}^{*}bU_{s}a_{0}(s))^{* }(a_{0}(s)U_{s}^{* }bU_{s}-U_{s}^{* }bU_{s}a_{0}(s)){y_n}(s),{y_n}(s) \rangle \\
& = \langle a_{0}(s)U_{s}^{* }bU_{s}{y_n}(s), a_{0}(s)U_{s}^{* }bU_{s}{y_n}(s)\rangle \\
&\  \quad\quad -  \langle a_{0}(s)U_{s}^{* }bU_{s}{y_n}(s),  U_{s}^{*}bU_{s}a_{0}(s){y_n}(s) \rangle \\
&\   \quad\quad -   \langle U_{s}^{* }bU_{s}a_{0}(s){y_n}(s),  a_{0}(s)U_{s}^{* }bU_{s}{y_n}(s) \rangle \\
&\   \quad\quad +  \langle U_{s}^{* }bU_{s}a_{0}(s){y_n}(s), U_{s}^{* }bU_{s}a_{0}(s){y_n}(s)\rangle \\
& =  \langle bU_{s}{y_n}(s),U_{s}a_0^{* }(s)U_{s}^{* }U_{s}a_{0}(s)U_{s}^{* }bU_{s}{y_n}(s) \rangle \\
& \ \quad\quad -  \langle U_{s}a_{0}(s)U_{s}^{*}bU_{s}{y_n}(s), bU_{s}a_{0}(s)U_{s}^{*}U_{s}{y_n}(s) \rangle \\
& \ \quad\quad  -  \langle bU_{s}a_{0}(s)U_{s}^{*}U_{s}{y_n}(s), U_{s}a_{0}(s)U_{s}^{*}bU_{s}{y_n}(s) \rangle \\
& \ \quad\quad +  \langle U_{s}a_{0}(s)U_{s}^{*}U_{s}{y_n}(s),b^{*}bU_{s}a_{0}(s)U_{s}^{*}U_{s}{y_n}(s) \rangle.
\end{aligned}$$
By the choice of the family $\{U_{s}: s \in X \}$ in Remark \ref{2.4}, the maps
\begin{align}
&s \to U_{s}a_{0}(s)U_{s}^{*} \label{1}
\\
&s \to U_{s}a_{0}^{*}(s)U_{s}^{*} \label{2}
\end{align}
from $X$ to $B(K)$ and $$s \to U_{s}{y_n}(s)$$ from $X$ to $K$ are
measurable. Therefore by Lemma 14.3.1 in \cite{KR1}, there is a
Borel $\mu$-null subset $N_{n,1}$ of X such that, restricted to $X
\setminus N_{n,1}$, the maps (\ref{1}) and (\ref{2}) are all Borel
maps. It follows that  the map $\phi_n$ is a Borel map from $(X\setminus
N_{n,1})  \times \mathcal{B}$.

Since $\omega_{{y_n}(s)}(U_{s}^{*}bU_{s})=\langle bU_{s}{y_n}(s),
U_{s}{y_n}(s) \rangle$, the map $(s, b) \to
\omega_{{y_n}(s)}(U_{s}^{*}bU_{s})$ from $X \times \mathcal{B}$ to
$\mathbb{C}$ is measurable by the choice of $\{ U_{s}: s \in X \}$
in Remark \ref{2.4}. By Lemma 14.3.1 in \cite{KR1}, there exists a
Borel $\mu$-null subset $N_{n,2}$ of $X$ such that the map
\begin{eqnarray*}
\psi_n: (s, b) \to \omega_{{y_n}(s)}(U_{s}^{*}bU_{s})
\end{eqnarray*}
is Borel measurable from $(X \setminus N_{n,2}) \times \mathcal{B}$ to $\mathbb{C}$.

By the  discussions in the preceding paragraphs, the maps
$$\phi_n: (s,b)\rightarrow \omega_{y_{n}(s)}((a_{0}(s)U_{s}^{*}bU_{s}-U_{s}^{*}bU_{s}a_{0}(s))^{*}(a_{0}(s)U_{s}^{*}bU_{s}-U_{s}^{*}bU_{s}a_{0}(s)))$$
and
$$\psi_n:  (s,b) \to \omega_{y_n(s)}(U_{s}^{*}bU_{s})$$
are Borel measurable from $(X\setminus N_n) \times \mathcal{B}$ to $\mathbb{C}$, where $ N_n=N_{n,1}\cup N_{n,2}$ is a $\mu$-null subset of $X$.

Since $\rho_{s}=\sum\limits_{n=1}^{\infty} \omega _{y_{n}(s)} $, we obtain that
\begin{equation}(s,b)\rightarrow \rho_{s}((a_{0}(s)U_{s}^{*}bU_{s}-U_{s}^{*}bU_{s}a_{0}(s))^{*}(a_{0}(s)U_{s}^{*}bU_{s}-U_{s}^{*}bU_{s}a_{0}(s)))\label{equa 2.5.2}\end{equation}
and
\begin{equation}(s,b) \to \rho_{s}(U_{s}^{*}bU_{s})\label{equa 2.5.3}\end{equation}
are Borel measurable from $(X\setminus N) \times \mathcal{B}$ to $\mathbb{C}$, where $N=\cup_{n\in \Bbb N}N_n$ is a Borel $\mu$-null subset of $X$. By (\ref{equa 2.5.1}),
(\ref{equa 2.5.2}), and (\ref{equa 2.5.3}), we  complete the proof of the proposition.
\end{proof}

\section{Property $\Gamma$ for type II$_{1}$ von Neumann algebras}

In this section, we will introduce Property $\Gamma$ for general von
Neumann algebras and discuss some of its properties.

 Murray and von Neumann's Property $\Gamma$ for a type II$_1$ factor
is defined as follows. {\em Suppose $\mathcal{A}$ is a type II$_{1}$
factor with a trace $\tau$. Let $\Vert \cdot \Vert_{2}$ be the
$2$-norm on $\mathcal{A}$ given by $\Vert a
\Vert_{2}=\sqrt{\tau(a^{*}a)}$ for any $a \in \mathcal{A}$. Then
$\mathcal{A}$ has Property $\Gamma$ if, given $\epsilon >0$ and
$a_{1}, a_{2}, \dots, a_{k} \in \mathcal{A}$, there exists a unitary
$u \in \mathcal{A}$ such that

\begin{enumerate}
\item [(a)] $\tau(u)=0$;

\item [(b)] $\Vert ua_{j}-a_{j}u \Vert_{2} < \epsilon, \ \  \forall \  1 \leq j
\leq k$.
\end{enumerate}}
An equivalent definition of Property $\Gamma$ for a type II$_{1}$
factor $\mathcal{A}$ was given by Dixmier in
\cite{Di1}. {\em Suppose $\mathcal{A}$ is a type II$_{1}$ factor with a trace $\tau$. Let
$\Vert \cdot \Vert_{2}$ be the $2$-norm on $\mathcal{A}$ given by
$\Vert a \Vert_{2}=\sqrt{\tau(a^{*}a)}$ for any $a \in
\mathcal{A}$. Then $\mathcal{A}$ has Property $\Gamma$ if, given $n
\in \mathbb{N}$, $\epsilon >0$ and $a_{1}, a_{2}, \dots, a_{k} \in
\mathcal{A}$, there exists $n$ orthogonal equivalent projections $\{
p_{1}, p_{2}, \dots, p_{n} \}$ in $\mathcal{A}$ with sum $I$ such
that
$$\Vert p_{i}a_{j}-a_{j}p_{i} \Vert_{2} < \epsilon, \qquad  \forall \ 1
\leq i \leq n, 1 \leq j \leq k.$$}

 We introduce a  definition of Property $\hat\Gamma$ for a type
II$_{1}$ von Neumann algebra as follows.

\begin{definition} \label{3.1}
Suppose $\mathcal{M}$ is a type II$_{1}$ von Neumann algebra with a
predual $\mathcal M_{\sharp}$. Suppose that $\sigma (\mathcal M,
\mathcal M_\sharp)$ is the weak-$*$ topology on $\mathcal M$ induced
from $\mathcal M_\sharp$. We say that $\mathcal{M}$ has Property
$\hat{\Gamma}$ if and only if $ \forall \ a_{1}, a_{2}, \dots, a_{k} \in
\mathcal{M}$ and $\forall \ n\in \Bbb N$, there exist a partially ordered set $\Lambda$ and a family of
projections $$\{ p_{i \lambda}: 1\le i\le n; \lambda \in \Lambda \}\subseteq
\mathcal{M}$$ satisfying
\begin{enumerate}
\item [(i)] For each $\lambda \in \Lambda$,   $\{ p_{1 \lambda}, p_{2 \lambda}, \ldots,
p_{n \lambda} \}$ is a family of orthogonal equivalent projections in
$\mathcal M$ with sum $I$.
\item [(ii)] For each $1\le i\le n$ and $1\le j\le k$,
$$
\lim_{\lambda} (p_{i \lambda}a_{j}-a_{j}p_{i
\lambda})^*(p_{i \lambda}a_{j}-a_{j}p_{i \lambda}) =0\qquad \text {in $\sigma(\mathcal
M, \mathcal M_\sharp)$ topology.}$$
\end{enumerate}
\end{definition}

\begin{remark} \label{rem 3.1} Suppose that $\mathcal M$ acts on a
Hilbert space $ H$. It is well-known that $\sigma(\mathcal M, \mathcal M_\sharp)$, the weak-$*$ topology,
  on   the unit ball of
$\mathcal M$  coincides with the weak operator topology on
 the unit ball of
$\mathcal M$. (See Theorem 7.4.2 in \cite{KR1})

\end{remark}

 Let $\mathcal{M}$ be a countably decomposable  type II$_{1}$ von Neumann algebra with    a faithful normal tracial state $\rho$.
Let $\Vert \cdot \Vert_{2}$ be the $2$-norm on $\mathcal{M}$ given
by $\Vert a \Vert_{2}=\sqrt{\rho(a^{*}a)}, \forall a \in
\mathcal{M}$ . Let $H_{\rho}=L^{2}(\mathcal{M}, \rho)$. For each
element $a \in \mathcal{M}$, we denote by $\bar{a}$ the
corresponding vector in $H_{\rho}$. Let $\pi_{\rho}$ be the left
regular representation of $\mathcal{M}$ on $H_{\rho}$ induced by
$\pi_{\rho}(a)(\bar{b})=\bar{ab}, \forall a, b \in \mathcal{M}$. The
vector $\bar{I}$ is cyclic for $\pi_{\rho}$, where $I$ is the unit
of $\mathcal{M}$. Since that $\rho$ is faithful, we obtain that
$\pi_{\rho}$ is faithful.

The following result is well-known. For the purpose of completeness, we include a proof here.
\begin{lemma} \label{3.2}
Let $\mathcal{M}$ be a countably decomposable  type II$_{1}$ von
Neumann algebra acting on a Hilbert space $H$ and   $\rho$ be
a faithful normal tracial state on $\mathcal{M}$. Let $\Vert \cdot
\Vert_{2}$ be the $2$-norm on $\mathcal{M}$ given by $\Vert a
\Vert_{2}=\sqrt{\rho(a^{*}a)}, \forall a \in \mathcal{M}$. Then
the topology induced by $\Vert \cdot \Vert_{2}$ coincides with the
strong operator topology on bounded subsets of $\mathcal{M}$.
\end{lemma}
\begin{proof}
We claim that $\pi_{\rho}$ is $WOT-WOT$ continuous on bounded subsets of $\mathcal{M}$. To show this, we first suppose $\{ a_{\lambda} \}$ is a net in the unit ball $(\mathcal{M})_{1}$ of $\mathcal{M}$ such that
$$WOT-\lim\limits_{\lambda} a_{\lambda} =a \in (\mathcal{M})_{1}.$$
Then for any $b, c \in \mathcal{M}$,
\begin{eqnarray}
\lim\limits_{\lambda} \langle \pi_{\rho}(a_{\lambda}) \pi_{\rho}(b) \bar{I}, \pi_{\rho}(c) \bar{I} \rangle =\lim\limits_{\lambda} \rho(c^{*}a_{\lambda}b) = \rho(c^{*}ab)=\langle \pi_{\rho}(a) \pi_{\rho}(b) \bar{I}, \pi_{\rho}(c) \bar{I} \rangle. \label{3}
\end{eqnarray}
Since the vector $\bar{I}$ is cyclic for $\pi_{\rho}$, we obtain from (\ref{3}) that
$$\lim\limits_{\lambda} \langle \pi_{\rho}(a_{\lambda})x, y \rangle=\langle \pi_{\rho}(a)x, y \rangle,  \qquad \forall x, y \in H_{\rho}. $$
Therefore $WOT-\lim\limits_{\lambda} \pi_{\rho}(a_{\lambda}) = \pi_{\rho}(a)$ and $\pi_{\rho}$ is $WOT-WOT$ continuous on bounded subsets of $\mathcal{M}$.

Since  $(\mathcal{M})_{1}$ is $WOT$ compact, the unit ball
$(\pi_{\rho}(\mathcal{M}))_{1}=\pi_{\rho}((\mathcal{M})_{1})$ is $WOT$
closed. By Kaplansky's Density Theorem, $\pi_{\rho}(\mathcal{M})$ is
a von Neumann algebra. Hence $\pi_{\rho}$ from $\mathcal{M}$ onto $\pi_{\rho}(\mathcal{M})$ is a $*$-isomorphism between von Neumann algebras. By Theorem 7.1.16
in \cite{KR1}, $\pi_{\rho}$ is a $*$-homeomorphism from
$(\mathcal{M})_{1}$ onto $(\pi_{\rho}(\mathcal{M}))_{1}$ when both
are endowed with the strong operator topology.

Now we can prove the result. First suppose $\{ b_{\lambda} \}$ is a net in $(\mathcal{M})_{1}$ such that $SOT-\lim\limits_{\lambda} b_{\lambda}=0$. Then $SOT-\lim\limits_{\lambda} b_{\lambda}^{*}b_{\lambda}=0$. Since $\rho$ is $SOT$-continuous on $(\mathcal{M})_{1}$, $\lim\limits_{\lambda} \rho(b_{\lambda}^{*}b_{\lambda})=0$, which implies that $\lim\limits_{\lambda} \Vert b_{\lambda} \Vert_{2}=0$. On the other hand, suppose $\{ c_{\lambda} \}$ is a net in $(\mathcal{M})_{1}$ such that $\lim\limits_{\lambda} \Vert c_{\lambda} \Vert_{2}=0$. Then for any $a \in \mathcal{M}$, $\lim\limits_{\lambda} \Vert c_{\lambda}a \Vert_{2}=0$, and hence
\begin{eqnarray*}
\lim\limits_{\lambda} \langle \pi_{\rho} (c_{\lambda}) \bar{a}, \pi_{\rho} (c_{\lambda}) \bar{a} \rangle&=&\lim\limits_{i} \langle \pi_{\rho} (a^{*}c_{\lambda}^{*}c_{\lambda}a) \bar{I}, \bar{I} \rangle \\
&=&\lim\limits_{\lambda}\rho(a^{*}c_{\lambda}^{*}c_{\lambda}a)\\
&=&\lim\limits_{\lambda} \Vert c_{\lambda}a \Vert_{2}^{2}\\
&=&0.
\end{eqnarray*}
Because $\{ \bar{a}: a \in \mathcal{M} \}$ is norm dense in $H_{\rho}$, it follows that $SOT-\lim\limits_{i} \pi_{\rho} (c_{\lambda}) = 0 $ in $B(H_{\rho})$. Since $\pi_{\rho}$ is a homeomorphism from $(\mathcal{M})_{1}$ onto $(\pi_{\rho}(\mathcal{M}))_{1}$ when both are endowed with the strong operator topology, $SOT-\lim\limits_{i} c_{\lambda} = 0$ in $B(H)$.
\end{proof}

\begin{corollary} \label{3.3}
Let $\mathcal{M}$ be a countably decomposable  type II$_{1}$ von
Neumann algebra with a faithful normal tracial state $\rho$.
Then the following are equivalent:
\begin{enumerate}
\item [(a)] $\mathcal{M}$ has Property $\hat\Gamma$ (in the sense of Definition \ref{3.1});
\item [(b)] Given any $\epsilon >0$, any positive integer $n$ and elements $a_{1},
a_{2}, \dots, a_{k} \in \mathcal{M}$, there exist orthogonal
equivalent projections $ p_{1}, p_{2}, \dots, p_{n}$ in $\mathcal M$ summing to
$I$ satisfying
$$\Vert p_{i}a_{j}-a_{j}p_{i} \Vert_{2, \rho} < \epsilon, \qquad 1 \leq i \leq n, 1 \leq j \leq k,$$
where the $2$-norm $\Vert \cdot \Vert_{2, \rho}$ on $\mathcal{M}$ is given
by $\Vert a \Vert_{2, \rho}=\sqrt{\rho(a^{*}a)}$ for any $a \in
\mathcal{M}$.
\item [(c)] For any faithful normal tracial state $\tilde\rho$ on $\mathcal{M}$,
any $\epsilon >0$, any positive integer $n$ and elements $a_{1},
a_{2}, \dots, a_{k} \in \mathcal{M}$, there exist orthogonal
equivalent projections $ p_{1}, p_{2}, \dots, p_{n} $  in $\mathcal M$  summing to
$I$ satisfying
$$\Vert p_{i}a_{j}-a_{j}p_{i} \Vert_{2, \tilde{\rho}} < \epsilon, \qquad 1 \leq i \leq n, 1 \leq j \leq k,$$
where the $2$-norm $\Vert \cdot \Vert_{2, \tilde{\rho}}$ on $\mathcal{M}$ is given
by $\Vert a \Vert_{2, \tilde{\rho}}=\sqrt{\tilde\rho(a^{*}a)}$ for any $a \in
\mathcal{M}$.\end{enumerate}
\end{corollary}

\begin{proof}
We might assume that $\mathcal M$ acts on a Hilbert space $ H$.

(a)$\Rightarrow$(b) It follows directly from Definition \ref{3.1}, Remark \ref{rem 3.1} and Lemma \ref{3.2}.

(b)$\Rightarrow$(a) Assume that (b) holds. Let $n \in \Bbb N$ and
$a_{1}, a_{2}, \dots, a_{k} \in \mathcal{M}$. By (b), there exists a
family of projections
$$\{ p_{i r}: 1\le i\le n; r\ge 1 \}\subseteq \mathcal{M}$$
satisfying
\begin{enumerate}
\item [(1)] for each $r\ge 1$,   $p_{1r}, p_{2r}, \ldots,
p_{nr}$ is a family of orthogonal equivalent projections in
$\mathcal M$ with sum $I$.
\item [(2)] for each $1\le i\le n$ and $1\le j\le k$,
\begin{equation} \lim_{r\rightarrow \infty} \rho((p_{ir}a_{j}-a_{j}p_{i
r})^*(p_{ir}a_{j}-a_{j}p_{i r})) =0.\label {eq1}\end{equation}
\end{enumerate}
In order to show that $\mathcal M$ has Property $\hat{\Gamma}$, we need
only to verify that  the family of projections $\{ p_{i r}: 1\le
i\le n; r\ge 1 \}$ satisfies condition (ii) in Definition \ref{3.1}.
Actually, combining equation (\ref{eq1})  and Lemma \ref{3.2}, we
know that, for each $1 \le i \le n$ and $1 \le j \le k$, as $r\rightarrow \infty$, $ p_{ir}a_{j}-a_{j}p_{i r}$
converges to $0$ in strong operator topology. Therefore, for each $1 \le i \le n$ and $1 \le j \le k$, as
$r\rightarrow \infty$, $(p_{ir}a_{j}-a_{j}p_{i
r})^*(p_{ir}a_{j}-a_{j}p_{i r})$ converges to $0$ in weak operator
topology and, whence in $\sigma(\mathcal M, \mathcal M_\sharp)$ by
Remark \ref{rem 3.1}. Therefore, $\mathcal M$ has Property $\hat{\Gamma}$.

 (b)$\Leftrightarrow$(c) Suppose $\rho_{1}$ and $\rho_{2}$ are
two faithful normal tracial states on $\mathcal{M}$. By Lemma
\ref{3.2}, the $2$-norms induced by $\rho_{1}$ and $\rho_{2}$ will
give the same topology on bounded subsets of $\mathcal{M}$ (since
they both coincide with the strong operator topology on bounded
subsets of $\mathcal{M}$). Therefore (b) and (c) are equivalent.
\end{proof}

\begin{corollary}\label{3.3.5}
Suppose that $\mathcal M$ is a factor of type II$_1$ with a tracial state $\tau$. The following are equivalent:
\begin{enumerate}
 \item [(i)] $\mathcal M$ has Property $\hat\Gamma$ in the sense of Definition \ref{3.1}.
 \item [(ii)] $\mathcal M$ has Property $\Gamma$ in the sense of Dixmier (equivalently,  of Murray and von Neumann).
\end{enumerate}
\end{corollary}
\begin{proof}
A type II$_1$ factor  is countably decomposable and $\tau$ is the unique faithful normal tracial state of $\mathcal M$.
From Dixmier's Definition of Property $\Gamma$ for type II$_1$ factors and Corollary \ref{3.3}, we know that $(i) \Leftrightarrow (ii)$.
\end{proof}
\begin{remark}\label{3.3.6}
Because of Corollary \ref{3.3.5}, from now on  we will use Definition \ref{3.1} as a definition of Property $\Gamma$ for type II$_1$ von Neumann algebras.
\end{remark}

In the rest of the paper, we will only consider von Neumann algebras
with separable predual because direct integral theory is only
applied to von Neumann algebras with separable predual. Next
proposition follows directly from Definition \ref{3.1}, Corollary
\ref{3.3} and the assumption that $\mathcal{M}$ is a type II$_{1}$
von Neumann algebra with separable predual.

\begin{proposition} \label{3.4}
Let $\mathcal{M}$ be a type II$_{1}$ von Neumann algebra with
separable predual and  $\rho$ be   a faithful normal tracial state on
$\mathcal{M}$. Then $\mathcal{M}$ has Property $\Gamma$ if and only
if for any $n \in \mathbb{N}$, there exists a family of projections
$\{ p_{ir}: 1 \leq i \leq n, r \in \mathbb{N} \}$ such that
\begin{enumerate}
\item [(i)] for each $r \in \mathbb{N}$, $\{ p_{ir}: 1\leq i \leq n \}$ is a set of $n$ equivalent orthogonal projections in $\mathcal{M}$ with sum $I$;

\item [(ii)] for each $1 \leq i \leq n$, $\lim\limits_{r \to \infty} \Vert p_{ir}a-ap_{ir} \Vert_{2} = 0$  for any $a \in \mathcal{M}$, where $\Vert \cdot \Vert_{2}$ is the $2$-norm induced by $\rho$ on $\mathcal{M}$.\end{enumerate}
\end{proposition}

With the help of Proposition \ref{3.4} and Corollary \ref{3.3}, we can directly get the next corollary.

\begin{corollary} \label{3.5}
Let $\mathcal{M}$ be a type II$_{1}$ von Neumann algebra with
separable predual and $\rho$ a faithful normal tracial state on
$\mathcal{M}$. Suppose $\{ a_{j}: j \in \mathbb{N} \}$ is a sequence of elements
in $\mathcal{M}$ that generates $\mathcal{M}$ as a von Neumann
algebra. Then $\mathcal{M}$ has Property $\Gamma$ if and only if for
any $n \in \mathbb{N}$, there exists a family of projections $\{
p_{ir}: 1 \leq i \leq n, r \in \mathbb{N} \}$ such that
\begin{enumerate}
\item [(i)] for each $r \in \mathbb{N}$, $\{ p_{ir}: 1\leq i \leq n \}$ is a set of $n$ equivalent orthogonal projections in $\mathcal{M}$ with sum $I$;

\item [(ii)] for each $1 \leq i \leq n$ and $j \in \mathbb{N}$, $\lim\limits_{r \to \infty} \Vert p_{ir}a_{j}-a_{j}p_{ir} \Vert_{2} = 0$, where $\Vert \cdot \Vert_{2}$ is the $2$-norm induced by $\rho$ on $\mathcal{M}$.\end{enumerate}
\end{corollary}

\begin{example}\label{example 1}
Let $\mathcal{A}_{1}$ be a type II$_{1}$ factor with separable
predual and $\mathcal{A}_{2}$ a finite von
Neumann algebra with separable predual.
Suppose $\mathcal{A}_{1}$ has Property $\Gamma$. Then the von
Neumann algebra tensor product $\mathcal{A}_{1} \otimes
\mathcal{A}_{2}$ is a type II$_{1}$ von Neumann algebra with separable predual and Property
$\Gamma$.

\end{example}

\begin{remark} \label{3.6}
Let $\mathcal{M}$ be a von Neumann algebra acting on a separable
Hilbert space $H$ and $\mathcal{Z}$ the center of
$\mathcal{M}$. Suppose $\mathcal{M}=\int_{X} \bigoplus
\mathcal{M}_{s} d\mu$ and $H=\int_{X} \bigoplus H_{s} d\mu$ are the
direct integral decompositions of $\mathcal{M}$ and $H$ over $(X,
\mu)$ relative to $\mathcal{Z}$. Take a countable $SOT$ dense self-adjoint subset
$\mathcal{F}$ of $\mathcal{M}$ and let $\mathcal{S}$ be the set of
all rational $*$-polynomials (i.e, coefficients from
$\mathbb{Q}+i \mathbb{Q}$) with variables from $\mathcal{F}$. We
observe that $\mathcal{S}$ is countable and $SOT$ dense in
$\mathcal{M}$. Take $\{ a_{j}: j \in \mathbb{N} \}$ to be the unit
ball of $\mathcal{S}$. By Kaplansky's Density Theorem, $\{ a_{j}: j
\in \mathbb{N} \}$ is $SOT$ dense in the unit ball
$(\mathcal{M})_{1}$.  By Definition 14.1.14 and Lemma 14.1.15 in
\cite{KR1}, $\{ a_{j}(s): j \in \mathbb{N} \}$ is $SOT$ dense in the
unit ball $(\mathcal{M}_{s})_{1}$ for almost every $s \in X$. In the
rest of this paper, when we mention a $SOT$ dense sequence $\{
a_{j}: j \in \mathbb{N} \}$of $(\mathcal{M})_{1}$ (or
$(\mathcal{M}')_{1}$), we always assume that this sequence has been
chosen such that $\{ a_{j}(s): j \in \mathbb{N} \}$ is $SOT$ dense
in the unit ball $(\mathcal{M}_{s})_{1}$ (or
$(\mathcal{M}'_{s})_{1}$) for almost every $s \in X$.
\end{remark}

\begin{lemma} \label{3.8}
If $H= \int_{X} \bigoplus H_{s} d\mu$ is a direct integral of separable Hilbert spaces and $\mathcal{A}$ is a decomposable von Neumann algebra (see definition in \cite{KR1}) over $H$ such that $\mathcal{A} \cong M_{m}(\mathbb{C})$, the $m \times m$ matrix algebra over $\mathbb{C}$ for some $m \in \mathbb{N}$. Then $\mathcal{A}_{s} \cong M_{m}(\mathbb{C})$ for almost every $s \in X$.
\end{lemma}
\begin{proof}
Notice that $\mathcal{A}$ is also a finite dimensional C$^*$-algebra. By Theorem 14.1.13 in \cite{KR1} and the fact that $\mathcal{A}$ is separable, the map from
$\mathcal{A}$ to $\mathcal{A}_{s}$ given by $a \to a(s)$ is a unital
$*$-homomorphism for almost every $s \in X$. Since
$\mathcal{A} \cong M_{m}(\mathbb{C})$, $\mathcal{A}$ is simple.
Therefore $\mathcal{A}_{s} \cong M_{m}(\mathbb{C})$ for almost every
$s \in X$.
\end{proof}

The following Proposition gives a useful characterization of a type II$_1$ von Neumann algebra with Property $\Gamma$.
\begin{proposition} \label{3.7}
Let $\mathcal{M}$ be a type II$_{1}$ von Neumann algebra acting on a
separable Hilbert space $H$ and   $\mathcal{Z}$ the center of
$\mathcal{M}$.
Suppose  $\rho$ is a faithful normal tracial state on $\mathcal{M}$
and a $2$-norm $\Vert \cdot \Vert_{2}$ on $\mathcal{M}$ is defined by  $\Vert a \Vert_{2}=\sqrt{\rho(a^{*}a)}$ for any $a \in
\mathcal{M}$.
Suppose $\mathcal{M}=\int_{X} \bigoplus
\mathcal{M}_{s} d\mu$ and $H=\int_{X} \bigoplus H_{s} d\mu$ are the
direct integral decompositions of $\mathcal{M}$ and $H$ over $(X,
\mu)$ relative to $\mathcal{Z}$. Then
\begin{enumerate}
\item [(i)] $\mathcal{M}$ has Property
$\Gamma$.
\item [(ii)]There exists a positive integer $n_0\ge 2$ such that
\begin{enumerate} \item []
 for any  $  \epsilon >0$,  and any $  a_{1},
a_{2}, \dots, a_{k} \in \mathcal{M}$, there exist orthogonal
equivalent projections $ p_{1}$, $p_{2}$, $\dots,$ $ p_{n_0}$ in $\mathcal M$ summing to
$I$ satisfying
$$\Vert p_{i}a_{j}-a_{j}p_{i} \Vert_{2} < \epsilon, \qquad 1 \leq i \leq n_0, 1 \leq j \leq k.$$
\end{enumerate}
\item [(iii)] $\mathcal{M}_{s}$ is a type II$_{1}$ factor
with Property $\Gamma$ for almost every $s \in X$.

\end{enumerate}
\end{proposition}

\begin{proof}
Since $\mathcal{M}$ is a type II$_{1}$ von Neumann algebra, by Lemma
\ref{2.1},  the component $\mathcal{M}_{s}$ is a type II$_1$ factor
for almost every $s \in X$. We may assume $\mathcal{M}_{s}$ is a
type II$_{1}$ factor with a trace $\tau_{s}$ for each $s \in X$.

 By Lemma \ref{2.2}, there is a
positive faithful normal tracial linear functional $\rho_{s}$ on $\mathcal{M}_{s}$
for almost every $s \in X$ such that $\rho(a)=\int_{X} \rho_{s}(a(s))d\mu$
for each $a$ in $\mathcal{M}$. We may assume $\rho_{s}$ is positive, faithful, normal and tracial
for each $s \in X$. Hence for each $s \in X$, $\rho_{s}$ is a
positive scalar multiple of the unique trace $\tau_{s}$ on the type
II$_{1}$ factor $\mathcal{M}_{s}$.

Let $\{ a_{j}: j \in \mathbb{N} \}, \{ a'_{j}: j \in \mathbb{N} \}$
be $SOT$ dense  subsets of the unit balls $\mathcal{M}_1,
(\mathcal{M}')_1$ of $\mathcal{M}$ and $\mathcal{M}'$ respectively.
By Proposition 14.1.24 in \cite{KR1}, we may assume that
$(\mathcal{M}')_{s}=(\mathcal{M}_{s})'$ for every $s \in X$ and we
use the notation $\mathcal{M}'_{s}$ for both. By Remark \ref{3.6},
we may assume $\{ a_{j}(s): j \in \mathbb{N} \}$ and $\{ a'_{j}(s):
j \in \mathbb{N} \}$ are $SOT$ dense in $(\mathcal{M}_{s})_{1}$ and
$(\mathcal{M}'_{s})_{1}$ for every $s \in X$.

(i)$\Rightarrow$ (ii):  The result is clear from Corollary \ref{3.3}, Corollary \ref{3.3.5} and Remark \ref{3.3.6}.

(ii)$\Rightarrow$ (iii): For this direction, we suppose (ii) holds. Notice that $\mathcal M$ acts on a separable Hilbert space, whence $\mathcal M$ is countably generated in strong operator topology. By
(ii), there exists a sequence of systems of matrix units $\{ \{e_{i,j}^{(r)}\}_{   i , j=1}^{ n_0} \ : \ r \in \mathbb{N} \}$ such that
\begin{enumerate}
\item [(A)] for each $r \in \mathbb{N}$,  we have $$\text{$\sum_{i=1}^{n_0} e_{i,i}^{(r)}=I$, \quad $(e_{i,j}^{(r)})^*= e_{j,i}^{(r)}$ \quad and  \ \
$e_{i,j}^{(r)}e_{j,k}^{(r)}=e_{i,k}^{(r)}$ \ \ \ for all $1\le
i,j,k\le n_0$.}$$
\item [(B)] for each $1 \leq i \leq n_0$, $\lim\limits_{r \to \infty} \Vert e_{i,i}^{(r)}a-ae_{i,i}^{(r)} \Vert_{2}=0$  for any $a \in \mathcal{M}$.
\end{enumerate}
By condition (A) and Lemma \ref{3.8}, there exists a
 $\mu$-null subset $N_{0}$ of X such that, for each $r \in \mathbb{N}$, $ \{e_{i,j}^{(r)}(s)\}_{   i , j=1}^{ n_0} $ is a system of matrix units such that  $\sum_{i=1}^{n_0} e_{i,i}^{(r)}(s)=I_{s}$ (the identity in $\mathcal{M}_{s}$)  in $\mathcal{M}_{s}$ for each $s \in X \setminus N_{0}$. In the following, we let
 $$ p_{i,r}=e_{ii}^{(r)} \qquad \text {for all } 1\le i\le n_0, r\in \Bbb N.  $$ Therefore, without loss of generality,
we can  assume that \begin{enumerate}
\item [ (I)]
$\{ p_{1,r}(s), p_{2,r}(s), \dots, p_{n_0,r}(s) \}$ is a set of $n_0$ equivalent orthogonal projections with sum $I_{s}$ in $\mathcal{M}_{s}$ for every $r \in \mathbb{N}$ and every $s \in X$;
\item [ (II)] for each $1 \leq i \leq n_0$,
$$   \lim\limits_{r \to \infty} \Vert p_{i,r} a-ap_{i,r} \Vert_{2}=0  \qquad \text{  for any $a \in \mathcal{M}$}.
$$\end{enumerate}

In the following we will use   a diagonal selection process to  produce a subsequence $\{ r_{m}: m \in \mathbb{N} \}$ of $\{ r: r \in \mathbb{N} \}$ and a $\mu$-null subset $X_{0}$ of $X$ such that
\begin{eqnarray}
\lim\limits_{m \to \infty} \Vert p_{i,r_{m}}(s)a_{j}(s)-a_{j}(s)p_{i,r_{m}}(s) \Vert_{2,s}=0  \qquad \forall \ i \in \{ 1, 2, \dots, n_0\}  \text { and } \forall  \ s \in X \setminus X_{0},
\end{eqnarray}
where the $\Vert \cdot \Vert_{2,s}$ is the $2$-norm induced by the unique trace $\tau_{s}$ on each $\mathcal{M}_{s}$.

First, by Assumption (II), for each $i \in \{ 1, 2, \dots, n_0 \}$,
$$\lim\limits_{r \to \infty} \Vert p_{i,r}a_{1}-a_{1}p_{i,r} \Vert_{2}=\lim\limits_{r \to \infty} \int_{X} \rho_{s}((p_{i,r}(s)a_{1}(s)-a_{1}(s)p_{i,r}(s))^{*}(p_{i,r}(s)a_{1}(s)-a_{1}(s)p_{i,r}(s))) d\mu = 0.$$
Therefore there exists a $\mu$-null subset $Y_{1}$ of $X$ and a subsequence $\{ r_{1,m}: m \in \mathbb{N} \}$ of $\{ r: r \in \mathbb{N} \}$ such that
$$\lim\limits_{m \to \infty} \rho_{s}((p_{i,r_{1,m} }(s)a_{1}(s)-a_{1}(s)p_{i,r_{1,m}}(s))^{*}(p_{i,r_{1,m}}(s)a_{1}(s)-a_{1}(s)p_{1,r_{1,m}}(s))) = 0$$
for any $s \in X \setminus Y_{1} $ and any $i \in \{ 1, 2, \dots, n_0
\}$. Since $\rho_{s}$ is a positive scalar multiple of the unique
trace $\tau_{s}$ on the type II$_{1}$ factor $\mathcal{M}_{s}$, we
obtain
$$
\lim\limits_{m \to \infty}\Vert p_{i,r_{1,m}} (s)a_{1}(s)-a_{1}(s)p_{i,r_{1,m} }(s)\Vert_{2,s} = 0
$$
for any $i \in \{ 1, 2, \dots, n \}$ and any $s \in X \setminus Y_{1}$, where $\Vert \cdot \Vert_{2,s}$ is the $2$-norm on $\mathcal{M}_{s}$ induced by $\tau_{s}$.

Again, there is a subsequence $\{ r_{2,m} : m \in \mathbb{N} \}$ of $\{ r_{1,m} : m \in \mathbb{N} \}$ and a $\mu$-null subset $Y_{2}$ of $X$ such that
$$
\lim\limits_{m \to \infty} \Vert p_{i,r_{2,m}}(s)a_{2}(s)-a_{2}(s)p_{i,r_{2,m}}(s) \Vert_{2,s}=0
$$
for any $i \in \{ 1, 2, \dots, n_0 \}$ and any $s \in X\setminus Y_{2}$.

Continuing in this way, we  obtain a subsequence $\{ r_{k,m}: m \in \mathbb{N} \}$ of $\{ r_{k-1,m}: m \in \mathbb{N} \}$ and a $\mu$-null subset $Y_{k}$ for each $k\ge 2$, satisfying
$$
\lim\limits_{m \to \infty} \Vert p_{i,r_{k,m}}(s)a_{k}(s)-a_{k}(s)p_{i,r_{k,m}}(s) \Vert_{2,s}=0
$$
for any $i \in \{ 1, 2, \dots, n_0 \}$ and any $s \in X\setminus Y_{k}$. Now we apply the diagonal selection by letting $r_{m}=r_{m,m}$ for each $m \in \mathbb{N}$ to these subsequences and obtain that
\begin{align}
\lim\limits_{m \to \infty} \Vert p_{i,r_{m}}(s)a_{j}(s)-a_{j}(s)p_{i,r_{m}}(s) \Vert_{2,s} = 0 \label{8}
 \end{align}
for any $i \in \{ 1, 2, \dots, n_0 \}$, $j \in \mathbb{N}$ and $s \in X \setminus X_{0}$, where $X_{0}=\cup_{k \in \mathbb{N}} Y_{k}$ is a $\mu$-null subset of $X$.

Since $\{ a_{j}: j \in \mathbb{N} \}$ is $SOT$ dense in the unit ball of $\mathcal{M}_{s}$ for each $s \in X$, (\ref{8}) implies that, for any $i \in \{ 1, 2, \dots, n_0 \}$, $s \in X \setminus X_{0}$ and any $a \in \mathcal{M}_{s}$,
\begin{eqnarray}
\lim\limits_{m \to \infty} \Vert p_{i,r_{m}}(s)a-a p_{i,r_{m}}(s) \Vert_{2,s} = 0. \label{9}
\end{eqnarray}
It follows from (\ref{9}) and Assumption (I) that $\mathcal{M}_{s}$ is a type II$_{1}$
factor with Property $\Gamma$ for almost every $s \in X$.

(iii) $\Rightarrow$ (i):  Suppose $\mathcal{M}_{s}$ is a type
II$_{1}$ factor with Property $\Gamma$ for almost every $s \in X$.
We may assume that for every $s \in X$, $\mathcal{M}_{s}$ is a type
II$_{1}$ factor with Property $\Gamma$.

By Remark \ref{2.4}, we can obtain a separable Hilbert space $K$ and a family of unitaries $\{ U_{s}: H_{s} \to K;  s \in X \}$ such that $s \to U_{s}x(s)$ and $s \to U_{s}a(s)U_{s}^{*}$ are measurable for any $x \in H$ and any decomposable operator $a \in B(H)$.  Let $\mathcal{B}$ be the unit ball of self-adjoint elements in $B(K)$ equipped with the $*$-strong operator topology. Then it is metrizable by setting $d(S,T)=\sum_{m=1}^{\infty} 2^{-m} (\Vert (S-T)e_{m} \Vert + \Vert (S^{*}-T^{*})e_{m}  \Vert)$ for any $S, T \in \mathcal{B}$, where $\{ e_{m} \}$ is an orthonormal basis of $K$. The metric space $(\mathcal{B}, d)$ is complete and separable. Now let $\mathcal{B}_{1} =\mathcal{B}_{2} =\dots =\mathcal{B}_{l} =\dots =\mathcal{B}$ and $\mathcal{C} =\prod\limits_{ l \in \mathbb{N}} \mathcal{B}_{l}$ provided with the product topology of the  $*$-strong operator topology on each $B_{l}$. It follows that $\mathcal{C}$ is metrizable and it's also a complete separable metric space.

Replacing $a_{0}$ by $a_{j}$ for $j\in \Bbb N$,   we apply Proposition \ref{2.5} countably many times and obtain positive, faithful, normal, tracial linear functionals $\rho_s$ on  $\mathcal M_s$ (almost everywhere)  and a Borel $\mu$-null subset $N$ of $X$ such that,
\begin{enumerate}
\item [(1)] $\rho(a)=\int_X \rho_s(a(s))d\mu$ for every $a\in\mathcal M$;
\item [(2)] for any $j \in \mathbb{N}$, the maps
\begin{equation} s \to \rho_{s}((a_{j}U_{s}^{*}bU_{s}-U_{s}^{*}bU_{s}a_{j}(s))^{*}(a_{j}U_{s}^{*}bU_{s}-U_{s}^{*}bU_{s}a_{j}(s))) \label{equa3.10.1}\end{equation}
and
\begin{equation}s \to \rho_{s}(U_{s}^{*}bU_{s})\label{equa3.10.2}\end{equation}
from $X$ to $\mathbb{C}$ are Borel measurable when restricted to $X \setminus N$.
\end{enumerate}

We denote by $(s, (Q_{11}, Q_{21}, \dots, Q_{n1}, Q_{12}, Q_{22}, \dots, Q_{n2}, \dots))$   an element in $X\times \mathcal C$.
Since $b \to b^{*}$ and $b \to b^{2}$ are $*$-$SOT$ continuous from $\mathcal{B}$ to $\mathcal{B}$, the maps
\begin{align}
(s, (Q_{11}, Q_{21}, \dots, Q_{n1}, Q_{12}, Q_{22}, \dots, Q_{n2}, \dots)) &\to Q_{it}, \label{13}\\
(s, (Q_{11}, Q_{21}, \dots, Q_{n1}, Q_{12}, Q_{22}, \dots, Q_{n2}, \dots)) &\to Q_{it}^{2}, \label{14}\\
(s, (Q_{11}, Q_{21}, \dots, Q_{n1}, Q_{12}, Q_{22}, \dots, Q_{n2}, \dots)) &\to Q_{it}^{*} \label{15}
\end{align}
are Borel measurable from $X \times \mathcal{C}$ to $\mathcal{B}$.

By Remark \ref{2.4}, the map $s \to U_{s}a'_{j}(s)U_{s}^{*}$ from
$X$ to $\mathcal{B}$ is measurable for every $j \in \mathbb{N}$.
Therefore, by Lemma 14.3.1 in \cite{KR1}, there exists a Borel
$\mu$-null subset $N'$ of $X$ such that the map  $s \to
U_{s}a'_{j}(s)U_{s}^{*}$ is Borel measurable when restricted to
$X \setminus N'$ for every $j \in \mathbb{N}$. Hence the maps
\begin{align}
(s, (Q_{11}, Q_{21}, \dots, Q_{n1}, Q_{12}, Q_{22}, \dots, Q_{n2}, \dots)) &\to Q_{it}U_{s}a'_{j}(s)U_{s}^{*}, \label{16}\\
(s, (Q_{11}, Q_{21}, \dots, Q_{n1}, Q_{12}, Q_{22}, \dots, Q_{n2}, \dots)) &\to U_{s}a'_{j}(s)U_{s}^{*}Q_{it} \label{17}
\end{align}
are Borel measurable when restricted to $(X \setminus N') \times \mathcal{C}$ for every $j \in \mathbb{N}$.

Since the functionals $\rho_{s}$ are chosen such that the maps
$$s \to \rho_{s}((a_{j}U_{s}^{*}bU_{s}-U_{s}^{*}bU_{s}a_{j}(s))^{*}(a_{j}U_{s}^{*}bU_{s}-U_{s}^{*}bU_{s}a_{j}(s)))$$
and
$$s \to \rho_{s}(U_{s}^{*}bU_{s})$$
are Borel measurable when restricted to $X \setminus N$, where $N$ is a Borel $\mu$-null subset of $X$, the maps
\begin{eqnarray}
&&(s, (Q_{11}, Q_{21}, \dots, Q_{n1}, Q_{12}, Q_{22}, \dots, Q_{n2}, \dots))  \notag \\
&&\to \rho_{s}((a_{j}(s)U_{s}^{*}Q_{it}U_{s}-U_{s}^{*}Q_{it}U_{s}a_{j}(s))^{*}(a_{j}(s)U_{s}^{*}Q_{it}U_{s}-U_{s}^{*}Q_{it}U_{s}a_{j}(s))) \label{18}
\end{eqnarray}
and
\begin{eqnarray}
(s, (Q_{11}, Q_{21}, \dots, Q_{n1}, Q_{12}, Q_{22}, \dots, Q_{n2}, \dots)) \to \rho_{s}(U_{s}^{*}Q_{it}U_{s}) \label{19}
\end{eqnarray}
are Borel measurable when restricted to $(X\setminus N) \times \mathcal{C}$ for each $j \in \mathbb{N}$.

Take $N_{0}=N \cup N'$. Then we have the following claim.

{\bf Claim \ref{3.7}.1.} {\em $N_{0}$ is a Borel $\mu$-null subset of $X$ and the maps (\ref{13})-(\ref{19}) are Borel measurable when restricted to $X \setminus N_{0}$.}

Next we   introduce the following subset $\eta$ of  $(X\setminus N_{0}) \times \mathcal{C}$.

{\em Let $  \eta$ be a subset of $(X\setminus N_{0}) \times \mathcal{C}$ that consists of all these elements
$$(s, (Q_{11}, Q_{21}, \dots, Q_{n1}, Q_{12}, Q_{22}, \dots, Q_{n2}, \dots, Q_{1t}, Q_{2t}, \dots, Q_{nt}, \dots )) \in (X\setminus N_{0}) \times \mathcal{C}$$
satisfying
\begin{enumerate}
\item [(a)]  for any $1 \leq i \leq n, t \in \mathbb{N}$,
\begin{equation}
 Q_{it}=Q_{it}^{*}=Q_{it}^{2} \neq 0; \label{equa 3.7.5}
\end{equation}
\item [(b)]  for any $1 \leq i \leq n$, $t, j \in \mathbb{N}$,
\begin{equation}
 Q_{it}U_{s}a'_{j}(s)U_{s}^{*}=U_{s}a'_{j}(s)U_{s}^{*}Q_{it}; \label{equa 3.7.6}
\end{equation}
\item [(c)] for any $1 \leq i \leq n$, $t \in \mathbb{N}, 1 \leq j \leq t$,
\begin{equation}
 \rho_{s}((a_{j}(s)U_{s}^{*}Q_{it}U_{s}-U_{s}^{*}Q_{it}U_{s}a_{j}(s))^{*}(a_{j}(s)U_{s}^{*}
 Q_{it}U_{s}-U_{s}^{*}Q_{it}U_{s}a_{j}(s)))<1/t;  \label{equa 3.7.7}
\end{equation}
\item [(d)] for any $t \in \mathbb{N}$,
    \begin{equation}
 \rho_{s}(U_{s}^{*}Q_{1t}U_{s})= \dots =\rho_{s}(U_{s}^{*}Q_{nt}U_{s}) \ \text{ and } \ Q_{1t}+Q_{2t}+ \dots +Q_{nt}=I.  \label{equa 3.7.8}
\end{equation}
\end{enumerate}}

 We have the following claim.

{\bf Claim \ref{3.7}.2:} {\em The set $\eta$ is analytic.}

{Proof of Claim \ref{3.7}.2:} By Claim \ref{3.7}.1, we know the maps (\ref{13})-(\ref{19}) are
Borel measurable when restricted to $X \setminus N_{0}$. It follows that the set
$\eta$ is a Borel set. Thus by Theorem 14.3.5 in \cite{KR1}, $\eta$
is analytic. The proof of Claim \ref{3.7}.2 is completed.

{\bf Claim \ref{3.7}.3:} {\em Let $\pi$ be the projection of $X \times \mathcal{M}$ onto $X$. Then $\pi(\eta)= X \setminus N_{0}$.}

{Proof of  Claim \ref{3.7}.3:} Let
$$(s, (Q_{11}, Q_{21}, \dots, Q_{n1}, Q_{12}, Q_{22}, \dots, Q_{n2}, \dots, Q_{1t}, Q_{2t}, \dots, Q_{nt}, \dots ))$$
be an element in $\eta$. From the definitions of the set $ \eta$, it's not hard to see that condition (a) is equivalent to that each $Q_{it}$ is a nonzero projection. Since $\{ a'_{j}(s): j \in \mathbb{N} \}$ is $SOT$ dense in $(\mathcal{M}')_{1}$ for each $s \in X$, condition (b) is equivalent to the condition that $U_{s}^{*}Q_{it}U_{s} \in \mathcal{M}_{s}$. Notice that $\{ a_{j}(s): j \in \mathbb{N} \}$ is $SOT$ dense in $(\mathcal{M})_{1}$ for each $s \in X$, condition (c) is equivalent to
$$\lim\limits_{t \to \infty} \rho_{s}((aU_{s}^{*}Q_{it}U_{s}-U_{s}^{*}Q_{it}U_{s}a)^{*}(aU_{s}^{*}Q_{it}U_{s}-U_{s}^{*}Q_{it}U_{s}a))=0$$
for any $a \in \mathcal{M}_{s}$.  Furthermore, $\rho_{s}$ is a positive scalar multiple of $\tau_{s}$ on $\mathcal{M}_{s}$ for each $s \in X$, it follows that condition (c)  is equivalent to
$$\lim\limits_{t \to \infty} \Vert aU_{s}^{*}Q_{it}U_{s}-U_{s}^{*}Q_{it}U_{s}a \Vert_{2,s} = 0$$
for any $a \in \mathcal{M}_{s}$. Moreover,  $(s, (Q_{11}, Q_{21}, \dots, Q_{n1}, Q_{12}, Q_{22}, \dots, Q_{n2}, \dots, Q_{1t}, Q_{2t}, \dots, Q_{nt}, \dots ))$  satisfies condition (a) and condition (d) if and only if $U_{s}^{*}Q_{1t}U_{s}, U_{s}^{*}Q_{2t}U_{s}, \dots, U_{s}^{*}Q_{nt}U_{s}$ are $n$ equivalent projections in $\mathcal{M}_{s}$ with sum $I_{s}$ for each   $n\in \Bbb N$ and each $t \in \mathbb{N}$.

For each $s \in X$, notice that $\mathcal{M}_{s}$ is a type II$_{1}$ factor with Property
$\Gamma$. From the argument in the preceding paragraph,  there exist projections $\{
U_{s}^{*}Q_{it}U_{s}: 1 \leq i \leq n, t \in \mathbb{N}\}$ in
$\mathcal{M}_{s}$ such that
$$(s, (Q_{11}, Q_{21}, \dots, Q_{n1}, Q_{12}, Q_{22}, \dots, Q_{n2}, \dots, Q_{1t}, Q_{2t}, \dots, Q_{nt}, \dots ))  \in X \times \mathcal{C}$$
satisfies conditions (a), (b), (c) and (d). Therefore the image of $\eta$ under $\pi$ is exactly $X \setminus N_{0}$.
The proof of Claim \ref{3.7}.3 is completed.

\vspace{0.2cm}

\noindent(Continue the proof of Proposition \ref{3.7}:) By Claim \ref{3.7}.2 and Claim \ref{3.7}.3, $\eta$ is analytic and the image of $\eta$
under $\pi$ is $X \setminus N_{0}$. By Theorem 14.3.6 in \cite{KR1}, there is a measurable
mapping
$$s \to  (Q_{11}^{(s)}, Q_{21}^{(s)}, \dots, Q_{n1}^{(s)}, Q_{12}^{(s)}, Q_{22}^{(s)}, \dots, Q_{n2}^{(s)}, \dots)$$
from $X\setminus N_{0}$ to $\mathcal{C}$ such that,  for   $s \in X \setminus N_{0}$ almost everywhere,  $$(s, (Q_{11}^{(s)}, Q_{21}^{(s)}, \dots, Q_{n1}^{(s)}, Q_{12}^{(s)}, Q_{22}^{(s)}, \dots, Q_{n2}^{(s)}, \dots))$$ satisfies conditions (a), (b), (c) and (d) (see (\ref{equa 3.7.5}), (\ref{equa 3.7.6}), (\ref{equa 3.7.7}) and (\ref{equa 3.7.8})). By defining $Q_{it}^{(s)}=0$ for any $t \in \mathbb{N}, 1 \leq i \leq n, s \in N_{0}$, we obtain {\em a measurable mapping
\begin{equation}s \to  (Q_{11}^{(s)}, Q_{21}^{(s)}, \dots, Q_{n1}^{(s)}, Q_{12}^{(s)}, Q_{22}^{(s)}, \dots, Q_{n2}^{(s)}, \dots)\label{equa 3.7.4}\end{equation}
from $X$ to $\mathcal{C}$ such that  such that,   for   $s \in X $ almost everywhere,  $$(s, (Q_{11}^{(s)}, Q_{21}^{(s)}, \dots, Q_{n1}^{(s)}, Q_{12}^{(s)}, Q_{22}^{(s)}, \dots, Q_{n2}^{(s)}, \dots))$$ satisfies conditions (a), (b), (c) and (d) (see (\ref{equa 3.7.5}), (\ref{equa 3.7.6}), (\ref{equa 3.7.7}) and (\ref{equa 3.7.8})).}

By (\ref{equa 3.7.4}), for any $t \in \mathbb{N}, 1 \leq i \leq n$
and any vectors $y,z \in H$, we have
$$\langle U_{s}^{*}Q_{it}^{(s)}U_{s}y(s), z(s) \rangle = \langle Q_{it}^{(s)}U_{s}y(s), U_{s}z(s) \rangle$$
and thus the map
$$s \to \langle U_{s}^{*}Q_{it}^{(s)}U_{s}y(s), z(s) \rangle$$
is measurable. Since
$$\vert \langle U_{s}^{*}Q_{it}^{(s)}U_{s}y(s), z(s) \rangle \vert \leq \Vert y(s) \Vert \Vert z(s) \Vert,$$
the map $s \to \langle U_{s}^{*}Q_{it}U_{s}y(s), z(s) \rangle$ is
integrable. By Definition 14.1.1 in \cite{KR1}, it follows that
\begin{eqnarray}
U_{s}^{*}Q_{it}^{(s)}U_{s}y(s)=(p_{it}y)(s) \label{20}
\end{eqnarray}
almost everywhere for some $p_{it}y$ in $H$. For each $t \in \mathbb{N}$,  (\ref{20}) implies that $p_{it}(s)=U_{s}^{*}Q_{it}^{(s)}U_{s}$ for almost every $s \in X$. Therefore $p_{it} \in \mathcal{M}$ for each $t \in \mathbb{N}$. Notice conditions (a) and (d) together imply that $U_{s}^{*}Q_{1t}^{(s)}U_{s}, U_{s}^{*}Q_{2t}^{(s)}U_{s}, \dots, U_{s}^{*}Q_{nt}^{(s)}U_{s}$ are $n$ orthogonal equivalent projections in $\mathcal{M}_{s}$ with sum $I_{s}$ for each $t \in \mathbb{N}$. It follows that $p_{1t}, p_{2t}, \dots, p_{nt}$ are $n$ orthogonal equivalent projections in $\mathcal{M}$ with sum $I$ for each $t \in \mathbb{N}$.

In order to show that $\mathcal{M}$ has Property $\Gamma$, it suffices to show that for any $i \in \{ 1, 2, \dots, n \}$ and $a \in \mathcal{M}$,
\begin{eqnarray}
\lim\limits_{t \to \infty} \rho((ap_{it}-p_{it}a)^{*}(ap_{it}-p_{it}a)) =0. \label{21}
\end{eqnarray}
By condition (c), we obtain that for each  $j \in \mathbb{N}, 1\leq i \leq n$ and $s \in X$,
\begin{eqnarray}
\lim\limits_{t \to \infty} \rho_{s}((a_{j}(s)U_{s}^{*}Q_{it}^{(s)}U_{s}-U_{s}^{*}Q_{it}^{(s)}U_{s}a_{j}(s))^{*}(a_{j}(s)U_{s}^{*}Q_{it}^{(s)}U_{s}-U_{s}^{*}Q_{it}^{(s)}U_{s}a_{j}(s)))=0. \label{22}
\end{eqnarray}
Fix $i \in \{ 1, 2, \dots, n \}$ and $j \in \mathbb{N}$. For each $t \in \mathbb{N}$, define a function $f_{t}: X \to \mathbb{C}$ such that
$$f_{t}(s)=\rho_{s}((a_{j}(s)U_{s}^{*}Q_{is}^{(s)}U_{s}-U_{s}^{*}Q_{it}^{(s)}U_{s}a_{j}(s))^{*}(a_{j}(s)U_{s}^{*}Q_{it}^{(s)}U_{s}-U_{s}^{*}Q_{it}^{(s)}U_{s}a_{j}(s))).$$
It follows from (\ref{22}) that
\begin{eqnarray}
\lim\limits_{t \to \infty} f_{t}(s) = 0 \label{23}
\end{eqnarray}
almost everywhere. By Lemma 14.1.9 in \cite{KR1}, for each $j \in
\mathbb{N}$, $\Vert a_{j} \Vert$ is the essential bound of $\{ \Vert
a_{j}(s) \Vert: s \in X \}$. Therefore
\begin{eqnarray*}
&&\Vert (a_{j}(s)U_{s}^{*}Q_{it}^{(s)}U_{s}-U_{s}^{*}Q_{it}^{(s)}U_{s}a_{j}(s))^{*}(a_{j}(s)U_{s}^{*}Q_{it}^{(s)}U_{s}-U_{s}^{*}Q_{it}^{(s)}U_{s}a_{j}(s)) \Vert\\
&\leq& 4 \Vert a_{j}(s) \Vert ^{2}\\
&\leq& 4 \Vert a_{j} \Vert ^{2}
\end{eqnarray*}
almost everywhere. Hence
\begin{eqnarray}
0 \leq f_{t}(s) \leq 4 \Vert a_{j} \Vert^{2} \rho_{s}(I_{s}) \label{24}
\end{eqnarray}
almost everywhere. Furthermore,
\begin{eqnarray}
\int_{X} 4 \Vert a_{j} \Vert^{2} \rho_{s}(I_{s}) d\mu=4 \Vert a_{j} \Vert^{2} \rho(I)=4 \Vert a_{j} \Vert^{2} \leq 4, \label{25}
\end{eqnarray}
by the Dominated Convergence Theorem, it follows from (\ref{23}), (\ref{24}) and (\ref{25}) that
\begin{eqnarray}
\lim\limits_{t \to \infty} \int_{X} \rho_{s}((a_{j}(s)U_{s}^{*}Q_{it}^{(s)}U_{s}-U_{s}^{*}Q_{it}^{(s)}U_{s}a_{j}(s))^{*}(a_{j}(s)U_{s}^{*}Q_{it}^{(s)}U_{s}-U_{s}^{*}Q_{it}^{(s)}U_{s}a_{j}(s))) d \mu = 0. \label{26}
\end{eqnarray}
Since $p_{it}(s)=Q_{it}^{(s)}$ for almost every $s \in X$, (\ref{26}) implies
\begin{eqnarray}
\lim\limits_{t \to \infty} \rho((a_{j}p_{it}-p_{it}a_{j})^{*}(a(j)p_{it}-p_{it}a_{j})) =0. \label{27}
\end{eqnarray}
From the fact that $\{ a_{j}: j \in \mathbb{N} \}$ is $SOT$ dense in the unit
ball of $\mathcal{M}$, we obtain equation (\ref{21}) from
(\ref{27}). Thus $\mathcal{M}$ is a type II$_{1}$ von Neumann
algebra with Property $\Gamma$.
\end{proof}

If $\mathcal{M}$ is a type II$_{1}$ von Neumann algebra, then by
Lemma 6.5.6 in \cite{KR1}, for any $m \in \mathbb{N}$, there is a unital
subalgebra $\mathcal{A}$ of $\mathcal{M}$ such that $\mathcal{A}
\cong M_{m}(\mathbb{C})$.

\begin{proposition} \label{3.9}
Suppose $\mathcal{M}$ is a type II$_{1}$ von Neumann algebra acting
on a separable Hilbert space $H$. Suppose further $\mathcal{A}$ is a
unital subalgebra of $\mathcal{M}$ such that  $\mathcal{A}\cong
M_{m}(\mathbb{C})$ for some $m \in \mathbb{N}$. Let
$\mathcal{N}=\mathcal{A}' \cap \mathcal{M}$. Then $\mathcal{M}$ has
Property $\Gamma$ if and only if $\mathcal{N}$ has Property
$\Gamma$.
\end{proposition}

\begin{proof}
By Lemma 11.4.11 in \cite{KR1}, $\mathcal{M} \cong \mathcal{A}
{\otimes} \mathcal{N} \cong M_{m}(\mathbb{C}) {\otimes}
\mathcal{N}$. It is trivial to see that if $\mathcal{N}$ has
Property $\Gamma$, then $\mathcal{M}$ has Property $\Gamma$. Thus we
only need to show that Property $\Gamma$ of $\mathcal{M}$ implies
Property $\Gamma$ of $\mathcal{N}$.

Suppose $\mathcal{M}$ has Property $\Gamma$. Let $\mathcal{M}=
\int_{X} \bigoplus \mathcal{M}_{s}d\mu$ and $H=\int_{X} \bigoplus
H_{s} d\mu$ be the direct integral decompositions relative to the
center $\mathcal{Z}$ of $\mathcal{M}$. Since $\mathcal{M}$ is a type
II$_{1}$ von Neumann algebra with Property $\Gamma$, by Proposition
\ref{3.7}, $\mathcal{M}_{s}$ is a type II$_{1}$ factor with Property
$\Gamma$ for almost every $s \in X$. We may assume $\mathcal{M}_{s}$
is a type II$_{1}$ factor with Property $\Gamma$ for every $s \in
X$.

Since $\mathcal{A} \cong M_{m}(\mathbb{C})$, by Lemma \ref{3.8}, we
may assume $\mathcal{A}_{s} \cong M_{m}(\mathbb{C})$ for every $s
\in X$. Since $\mathcal{N}=\mathcal{A}' \cap \mathcal{M}$,
$\mathcal{N}_{s}=\mathcal{A}'_{s} \cap \mathcal{M}_{s}$ for almost
every $s \in X$. Then by Lemma 11.4.11 in \cite{KR1},
$$\mathcal{M}_{s} \cong \mathcal{A}_{s} {\otimes} \mathcal{N}_{s} \cong M_{m}(\mathbb{C}) {\otimes} \mathcal{N}_{s}$$
for almost every $s \in X$. Since $\mathcal{M}_{s}$ is a type
II$_{1}$ factor with Property $\Gamma$ for every $s \in X$, by Lemma
5.1 in \cite{EFAR2}, $\mathcal{N}_{s}$ has Property $\Gamma$ for
almost every $s \in X$. By a similar argument as the proof of
Proposition \ref{3.7}, we can conclude that $\mathcal{N}$ has
Property $\Gamma$.
\end{proof}

\section{Hyperfinite II$_1$ subfactors in type II$_1$ von Neumann algebras}

Let $\mathcal{M}$ be a type II$_{1}$ von Neumann algebra with
separable predual and Property $\Gamma$. We will devote  this
section to the construction of a hyperfinite type II$_{1}$ subfactor
$\mathcal{R}$ of $\mathcal{M}$ such that
\begin{enumerate}
\item [(I)] $\mathcal{R}' \cap \mathcal{M} = \mathcal{Z}$, where $\mathcal{Z}$ is the center of $\mathcal{M}$ ;

\item [(II)] for any given $a_{1}, a_{2}, \dots, a_{k} \in \mathcal{M}, n \in \mathbb{N}$ and $\epsilon >0$, there exist orthogonal equivalent projections $p_{1}, p_{2}, \dots, p_{n}$ in $\mathcal{R}$ with sum $I$ such that
$$\Vert p_{i}a_{j}-a_{j}p_{i} \Vert_{2} < \epsilon,  i=1, 2, \dots;  j=1, 2, \dots, k,$$
where the $2$-norm $\Vert \cdot \Vert_{2}$ is given by $\Vert a \Vert_{2}=\sqrt{\rho(a^{*}a)}, \forall a \in \mathcal{M}$ for some faithful tracial state $\rho$ on $\mathcal{M}$.
\end{enumerate}
\begin{lemma} \label{3.10}
Let $\mathcal{M}$ be a type II$_{1}$ von Neumann algebra acting on a
separable Hilbert space $H$.  Let $m \in \mathbb{N}$ and $\mathcal{A} $ be  a unital subalgebra of $\mathcal{M}$ such that $\mathcal{A} \cong
M_{m}(\mathbb{C})$. Let
$\mathcal{N}=\mathcal{A}' \cap \mathcal{M}$.
Assume that $\mathcal{M}=
\int_{X} \bigoplus \mathcal{M}_{s}d\mu$ and $H=\int_{X} \bigoplus
H_{s} d\mu$ are the direct integral decompositions relative to the
center $\mathcal{Z}$ of $\mathcal{M}$. Assume that $\rho$ is a faithful
normal tracial state on $\mathcal{M}$ and $\{ \rho_{s}: s \in X \}$ is a family of positive, faithful, normal, tracial functionals
as introduced in  Lemma \ref{2.2} and Proposition \ref{2.5}. If $\mathcal M$ has Property $\Gamma$, then
 \begin{enumerate}\item []
 $\forall a_{1}, a_{2}, \dots, a_{k} \in \mathcal{M}$, $\forall n
\in \mathbb{N}$ and $\forall \epsilon >0$, there exist  a $\mu$-null subset
$X_{0}$ of $X$ and a family of mutually orthogonal equivalent projections $\{p_{1},
p_{2}, \dots, p_{n}\}$ in $\mathcal{N}$ with sum $I$ such that,
    $$\rho_{s}((p_{i}(s)a_{j}(s)-a_{j}(s)p_{i}(s))^{*}(p_{i}(s)a_{j}(s)-a_{j}(s)p_{i}(s))) < \epsilon,$$
     for all $i=1,2, \dots, n,$ $ j=1,2, \dots, k,$ and $ s \in X \setminus X_{0}$.\end{enumerate}
\end{lemma}

\begin{proof}
Since $\mathcal{A} \cong M_{m}(\mathbb{C})$ and
$\mathcal{N}=\mathcal{A}' \cap \mathcal{M}$, by Lemma 11.4.11 in
\cite{KR1}, $\mathcal{M} \cong \mathcal{A} {\otimes} \mathcal{N}$.
Then by the discussion in section 11.2 in  \cite{KR1}, $\mathcal{N}$
is a type II$_{1}$ von Neumann algebra. Let $\{ a_{gh}
\}_{g,h=1}^{m}$ be a system of matrix units for $\mathcal{A}$. By
Lemma \ref{3.8}, we may assume that $\mathcal{A}_{s} \cong
M_{m}(\mathbb{C})$ and $\mathcal{M}_{s} \cong \mathcal{A}_{s}
{\otimes} \mathcal{N}_{s}$  for every $s \in X$. For each $s \in X$,
let $\{ a_{gh}(s)\}_{g,h=1}^{m}$ be a system of matrix units for
$\mathcal{A}_{s}$. By Proposition \ref{3.7}, we may assume that
$\mathcal{M}_{s}$ is a type II$_{1}$ factor with Property $\Gamma$
for every $s \in X$. Then by Lemma 5.1 in  \cite{EFAR2},
$\mathcal{N}_{s}$ is a type II$_{1}$ factor with Property $\Gamma$
for every $s \in X$.

Let $\{ a'_{r}: r \in \mathbb{N} \}$ be a $SOT$ dense subset in the
unit ball $\mathcal{M}'_{1}$ of $\mathcal{M}'$.  By Proposition
14.1.24 in \cite{KR1}, we may assume that
$(\mathcal{M}')_{s}=(\mathcal{M}_{s})'$ for every $s \in X$ and we
use the notation $\mathcal{M}'_{s}$ for both. Therefore by Remark
\ref{3.6}, we may assume $\{ a'_{r}(s) : r \in \mathbb{N} \}$ is
$SOT$ dense in $(\mathcal{M}_{s}')_{1}$ for every $s \in X$.

Take a separable Hilbert space $K$ and a family of unitaries $\{
U_{s}: H_{s} \to K; s \in X \}$ as in Remark \ref{2.4} such that $s
\to U_{s}x(s)$ and $s \to U_{s}a(s)U_{s}^{*}$ are measurable for
any $x \in H$ and any decomposable operator $a \in B(H)$.  Let
$\mathcal{B}$ be the unit ball of $B(K)$ equipped with the
$*$-strong operator topology. Since $K$ is separable,
$\mathcal{B}$ is metrizable by setting $d(S,T)=\sum_{j=1}^{\infty}
2^{-j}(\Vert (S-T)e_{j} \Vert +\Vert (S^{*}-T^{*})e_{j} \Vert
)$ for any $S, T \in \mathcal{B}$, where $\{ e_{j}: j \in \mathbb{N}
\}$ is an orthonormal basis for $K$. Then the metric space
$(\mathcal{B}, d)$ is complete and separable. For each $1 \leq i, j
\leq n$, let $\mathcal{B}_{ij}=\mathcal{B}$. Take
$\mathcal{C}=\prod\limits_{1 \leq i, j \leq n} \mathcal{B}_{ij}$
equipped with the product topology. It follows that $\mathcal{C}$ is a complete
separable metric space.

By the choices of $\{U_s\}$, we know that the maps $ s \to
U_{s}a'_{r}(s)U_{s}^{*},$ and $s \to U_{s}a_{gh}(s)U_{s}^{*} $ from
$X$ to $B(K)$ are measurable for any $r \in \mathbb{N}$ and any $ g,
h=1, 2, \dots, m$. By Lemma 14.3.1 in \cite{KR1}, there exists a
Borel $\mu$-null subset $N_{1}$ of $X$ such that the maps
\begin{align}
(s, b) &\to bU_{s}a'_{r}(s)U_{s}^{*}, \label{28} \\
(s, b) &\to U_{s}a'_{r}(s)U_{s}^{*}b, \label{29} \\
(s, b) &\to  bU_{s}a_{gh}(s)U_{s}^{*}, \label{30} \\
(s, b) &\to  U_{s}a_{gh}(s)U_{s}^{*}b \label{31}
\end{align}
are Borel measurable from $(X \setminus N_{1}) \times \mathcal{B}$ to $B(K)$ for any
$r \in \mathbb{N}$ and any $ g, h=1,  2, \dots, m$. Since $\rho$ is a faithful normal tracial state, by Lemma \ref{2.2},
we may assume that, for every $s\in X$, there exists a positive, faithful, normal,  tracial functional $\rho_{s}$ on  $\mathcal{M}_{s}$ such that $\rho (a)= \int_{X} \rho_{s}(a(s)) d\mu$ for any $a \in \mathcal{M}$. By Proposition \ref{2.5}, there is a Borel $\mu$-null subset $N_{2}$ of $X$ such that, for each $j \in \mathbb{N}$, the map
\begin{eqnarray}
(s, b) \to \rho_{s}((U_{s}^{*}bU_{s}a_{j}(s)-a_{j}(s)U_{s}^{*}bU_{s})^{*}(U_{s}^{*}bU_{s}a_{j}(s)-a_{j}(s)U_{s}^{*}bU_{s})) \label{32}
\end{eqnarray}
is Borel measurable from $( X\setminus N_{2}) \times \mathcal{B}$ to $\mathbb{C}$. From the fact that each $\rho_s$ is a positive, faithful, normal,  tracial functional  on   $\mathcal{M}_{s}$, it follows that $\rho_{s}$ is a positive scalar multiple of the unique trace $\tau_{s}$ on $\mathcal{M}_{s}$ for each $s$ in $X\setminus N_{2}$.

Let $N=N_{1} \cup N_{2}$.  {\em Let  $\eta$ be the collection of all these  elements $$(p, E_{11}, E_{12}, \dots, E_{nn}) \in (X\setminus N) \times \mathcal{C}$$ such that
\begin{enumerate}
\item [(i)]  for all $i_{1}, i_{2}, i_{3} \in \{ 1,2, \dots, n \}$,
\begin{equation} E_{i_{1}i_{2}}=E_{i_{2}i_{1}}^{*}\quad \text{and} \quad E_{i_{1}i_{2}}E_{i_{2}i_{3}}=E_{i_{1}i_{3}};\label{equa 3.10.2}\end{equation}

\item [(ii)]
\begin{equation}  E_{11}+E_{22}+\dots +E_{nn}=I;\label{equa 3.10.3}\end{equation}

\item [(iii)]  for all $i_{1}, i_{2} \in \{ 1,2, \dots, n \}$ and $r \in \mathbb{N},$
\begin{equation} E_{i_{1}i_{2}}U_{s}a'_{r}(s)U_{s}^{*}=U_{s}a'_{r}(s)U_{s}^{*}E_{i_{1}i_{2}} ;\label{equa 3.10.4}\end{equation}

\item [(iv)]  for all $i_{1}, i_{2} \in \{ 1,2, \dots, n \}$ and $g, h \in \{ 1,2, \dots, m \}$,
\begin{equation} E_{i_{1}i_{2}}U_{s}a_{gh}(s)U_{s}^{*}=U_{s}a_{gh}(s)U_{s}^{*}E_{i_{1}i_{2}}; \label{equa 3.10.5}\end{equation}

\item [(v)]  for all $i$ in $\{1,2, \dots, n\}$ and $j$ in $\{1,2, \dots, k\}$,
    \begin{equation}  \rho_{s}((U_{s}^{*}E_{ii}U_{s}a_{j}(s)-a_{j}(s)U_{s}^{*}E_{ii}U_{s})^{*}(U_{s}^{*}E_{ii}U_{s}a_{j}(s)-a_{j}(s)U_{s}^{*}E_{ii}U_{s})) < \epsilon. ;\label{equa 3.10.6}\end{equation}
\end{enumerate}}
We have the following claim.

\vspace{0.2cm}

{\bf Claim \ref{3.10}.1.} {\em The set $\eta$ is analytic.}

\vspace{0.2cm}

{\noindent Proof of Claim \ref{3.10}.1:} The maps
$$\begin{aligned}
 (E_{11}, E_{12}, \dots, E_{nn}) &\to E_{i_{1}i_{2}},\\
 (E_{11}, E_{12}, \dots, E_{nn}) &\to E_{i_{2}i_{1}}^{*},\\
 (E_{11}, E_{12}, \dots, E_{nn}) &\to E_{i_{1}i_{2}}E_{i_{2}i_{3}},\\
 (E_{11}, E_{12}, \dots, E_{nn}) &\to E_{11}+E_{22}+ \dots + E_{nn}
\end{aligned}
$$
are continuous from $\mathcal{C}$ (with the product topology) to $\mathcal{B}$ (with the $*$-strong operator topology). Therefore, we obtain that the maps
$$\begin{aligned}
 (s,E_{11}, E_{12}, \dots, E_{nn}) &\to E_{i_{1}i_{2}},\\
 (s,E_{11}, E_{12}, \dots, E_{nn}) &\to E_{i_{2}i_{1}}^{*},\\
 (s,E_{11}, E_{12}, \dots, E_{nn}) &\to E_{i_{1}i_{2}}E_{i_{2}i_{3}},\\
 (s,E_{11}, E_{12}, \dots, E_{nn}) &\to E_{11}+E_{22}+ \dots + E_{nn}
\end{aligned}
$$
are Borel measurable from $(X \setminus N) \times \mathcal{C}$ to $\mathcal{B}$ for all $1 \leq i_{1},i_{2},i_{3} \leq n$. From the fact that the maps (\ref{28}), (\ref{29}), (\ref{30}) and (\ref{31}) are Borel measurable from $(X \setminus N_{1}) \times \mathcal{B}$ to $B(K)$ and the map (\ref{32}) is Borel measurable from $(X \setminus N_{2}) \times \mathcal{B}$ to $\mathbb{C}$, it follows that the following maps
$$\begin{aligned}
(s, E_{11}, E_{12}, \dots, E_{nn}) &\to E_{i_{1}i_{2}}U_{s}a'_{r}(s)U_{s}^{*},\\
 (s, E_{11}, E_{12}, \dots, E_{nn}) &\to U_{s}a'_{r}(s)U_{s}^{*}E_{i_{1}i_{2}},\\
 (s, E_{11}, E_{12}, \dots, E_{nn}) &\to E_{i_{1}i_{2}}U_{s}a_{gh}(s)U_{s}^{*},\\
 (s, E_{11}, E_{12}, \dots, E_{nn}) &\to U_{s}a_{gh}(s)U_{s}^{*}E_{i_{1}i_{2}},\\
 (s, E_{11}, E_{12}, \dots, E_{nn}) &\to \rho_{s}((U_{s}^{*}E_{ii}U_{s}a_{j}(s)-a_{j}(s)U_{s}^{*}E_{ii}U_{s})^{*}(U_{s}^{*}E_{ii}U_{s}a_{j}(s)-a_{j}(s)U_{s}^{*}E_{ii}U_{s}))
\end{aligned}
$$
are Borel measurable
 when restricted to $(X \setminus N) \times \mathcal{C}$ for all  $1 \leq i_{1}, i_{2}, i \leq n,  1 \leq g,h \leq m, r \in \mathbb{N},$ and $ 1 \leq j \leq k$.
 Therefore $\eta$ is a Borel set. Thus $\eta$ is analytic by Theorem 14.3.5 in \cite{KR1}.This completes the proof of Claim \ref{3.10}.1.

 \vspace{0.2cm}
{\bf Claim \ref{3.10}.2.} {\em  Let $\pi$ be the projection of $X \times \mathcal{C}$ onto $X$. Then $\pi(\eta)= X \setminus N$.}

 \vspace{0.2cm}

{\noindent  Proof of Claim \ref{3.10}.2:} Notice that an element $(s, E_{11}, E_{12}, \dots,
E_{nn})$ in $(X \setminus N) \times \mathcal{C}$ satisfies conditions (i) and (ii) if and only if $\{
E_{i_{1}i_{2}} \}_{i_{1},i_{2}=1}^{n}$ is a system of matrix units for a matrix algebra which is isomorphic to $\mathcal M_n(\Bbb C)$.
Condition (iii) is equivalent to that $U_{s}^{*
}E_{i_{1}i_{2}}U_{s} \in \mathcal{M}_{s}$. Condition (iv) is
equivalent to that $U_{s}^{*}E_{i_{1}i_{2}}U_{s} \in
\mathcal{A}'_{s}$.

By assumption, for each $s \in X$, $\mathcal{M}_{s}$ and
$\mathcal{N}_{s}$ are type II$_{1}$ factors with Property $\Gamma$
and $\mathcal{M}_{s} \cong \mathcal{A}_{s} {\otimes}
\mathcal{N}_{s}$. Thus $\mathcal{A}'_{s} \cap \mathcal{M}_{s} =
\mathcal{N}_{s}$. It follows from the argument in the preceding
paragraph that,   for each $s \in X $, there exists a system of
matrix units  $\{ U_{s}^{*}E_{11}U_{s}, U_{s}^{*}E_{12}U_{s}, \dots,
U_{s}^{*}E_{nn}U_{s} \}$ in $\mathcal{N}_{s}$ such that $(s, E_{11},
E_{12}, \dots, E_{nn})$ satisfies conditions (i), (ii), (iii), (iv)
and (v). Therefore the image of $\eta$ under $\pi$ is exactly $X
\setminus N$. This completes the proof of Claim \ref{3.10}.2.

 \vspace{0.2cm}

 \noindent (Continue the proof of Lemma \ref{3.10}) By Claim \ref{3.10}.1 and Claim \ref{3.10}.2, $\eta$ is analytic and $\pi(\eta)=X\setminus
N$. By the measure-selection principle (Theorem 14.3.6 in
\cite{KR1}), there is a measurable mapping
$$s \to (E_{11,s}, E_{12,s}, \dots, E_{nn,s})$$
from $X \setminus N$ to $\mathcal{C}$ such that,  for  $s \in X \setminus N$ almost everywhere,
 $(s, E_{11,s}, E_{12,s}, \dots, E_{nn,s})$ satisfies conditions (i), (ii), (iii), (iv) and (v) (see (\ref {equa 3.10.2}),
  (\ref  {equa 3.10.3}), (\ref  {equa 3.10.4}), (\ref  {equa 3.10.5}), and (\ref  {equa 3.10.6})).
  Defining $E_{i_{1}i_{2}, s}=0$ for $s \in N, 1 \leq i_{1}, i_{2} \leq n$,  we get {\em a measurable map \begin{equation}s
  \to (E_{11,s}, E_{12,s}, \dots, E_{nn,s})\label {equa 3.10.1} \end{equation} from $X$ to $\mathcal{C}$ such that,  for  $s \in X $ almost everywhere, $(s, E_{11,s}, E_{12,s}, \dots, E_{nn,s})$ satisfies conditions (i), (ii), (iii), (iv) and (v)  (see (\ref {equa 3.10.2}), (\ref  {equa 3.10.3}), (\ref  {equa 3.10.4}), (\ref  {equa 3.10.5}), and (\ref  {equa 3.10.6})).}

From (\ref{equa 3.10.1}), for any $1 \leq i_{1}, i_{2} \leq n$ and any two vectors $x, y \in H$, it follows
$$\langle U_{s}^{*}E_{i_{1}i_{2},s}U_{s}x(s), y(s)\rangle=\langle E_{i_{1}i_{2},s}U_{s}x(s), U_{s}y_{s} \rangle,$$
and the map $s \to \langle U_{s}^{*}E_{i_{1}i_{2},s}U_{s}x(s), y(s)\rangle$ are measurable. Since
$$\vert \langle U_{s}^{*}E_{i_{1}i_{2}, s}U_{s}x(s), y(s)\rangle \vert \leq \Vert x(s) \Vert \Vert y(s) \Vert,$$
we know $s \to \langle U_{s}E_{i_{1}i_{2}, s}U_{s}x(s), y(s) \rangle$ is
integrable. By Definition 14.1.1 in \cite{KR1}, it follows that
\begin{eqnarray}
U_{s}^{*}E_{i_{1}i_{2}, s}U_{s}x(s)=(p_{i_{1}i_{2}}x)(s) \label{33}
\end{eqnarray}
almost everywhere for some $p_{i_{1}i_{2}}x \in H$. From (\ref{33}), we have that
\begin{eqnarray}
p_{i_{1}i_{2}}(s)=U_{s}^{*}E_{i_{1}i_{2}, s}U_{s} \label{34}
\end{eqnarray}
for almost every $s \in X$ and thus $p_{i_{1}i_{2}} \in \mathcal{M}$. By condition (iv),
$$U_{s}^{*}E_{i_{1}i_{2}, s}U_{s} \in \mathcal{A}_{s}'.$$
Hence
\begin{eqnarray}
p_{i_{1}i_{2}} \in \mathcal{A}' \cap \mathcal{M} = \mathcal{N}. \label{35}
\end{eqnarray}

Since conditions (i) and (ii) together imply that $\{ E_{i_{1}i_{2}, s} \}_{i_{1},i_{2}=1}^{n}$ is a system of matrix units, by (\ref{34}), we obtain that $p_{11}(s), p_{22}(s), \dots, p_{nn}(s)$ are $n$ orthogonal equivalent projections in $\mathcal{M}_{s}$ with sum $I_{s}$ almost everywhere. Therefore (\ref{35}) implies that $p_{11}, p_{22}, \dots, p_{nn}$ are $n$ orthogonal equivalent projections in $\mathcal{N}$ with sum $I$. For each $i \in \{ 1, 2, \dots, n\}$, let $p_{i}=p_{ii}$. From condition (v), we conclude that $p_{1}, p_{2}, \dots, p_{n}$ is a family of mutually orthogonal equivalent projections in $\mathcal N$  with sum $I$ satisfying,
 $\forall i=1,2, \dots, n, \forall j=1,2, \dots, k,$
    $$\rho_{s}((p_{i}(s)a_{j}(s)-a_{j}(s)p_{i}(s))^{*}(p_{i}(s)a_{j}(s)-a_{j}(s)p_{i}(s))) < \epsilon  \quad  \text { for }  s \in X  \ \text{ almost everywhere}.$$
    This ends the proof of the lemma.
\end{proof}

A slight modification of the proof in Lemma \ref{3.10} gives us the next corollary.

\begin{corollary} \label{3.11}
Let $\mathcal{M}$ be a type II$_{1}$ von Neumann algebra acting on a
separable Hilbert space $H$.  Let $m \in \mathbb{N}$ and  $\mathcal{A}$ be a unital subalgebra of $\mathcal{M}$ such that $\mathcal{A} \cong
M_{m}(\mathbb{C})$. Let $\mathcal{N}=\mathcal{A}' \cap \mathcal{M}$. Assume that  $\mathcal{M}=
\int_{X} \bigoplus \mathcal{M}_{s}d\mu$ and $H=\int_{X} \bigoplus
H_{s} d\mu$ are the direct integral decompositions relative to the
center $\mathcal{Z}$ of $\mathcal{M}$. Assume that   $\mathcal{M}$ has Property
$\Gamma$. Then,  $\forall a_{1}, a_{2}, \dots,
a_{k} \in  \mathcal{M}$, $\forall n \in \mathbb{N}$ and $\forall \epsilon >0$, there
exist  a $\mu$-null subset $X_{0}$ of $X$ and a family of mutually orthogonal
equivalent projections $\{p_{1}, p_{2}, \dots, p_{n}\}$ in
$\mathcal{N}$ with sum $I$ such that
 $$\Vert p_{i}(s)a_{j}(s)-a_{j}(s)p_{i}(s)) \Vert_{2,s} < \epsilon , \qquad \forall i=1,2, \dots, n, \ \forall j=1,2, \dots, k \text  { and } s \in X \setminus X_{0},$$ where $\Vert \cdot \Vert_{2,s}$ is the $2$-norm induced by the unique trace $\tau_{s}$ on $\mathcal{M}_{s}$.
\end{corollary}

In \cite{P}, Popa proved that if $\mathcal{A}$ is a type II$_{1}$
factor with separable predual, then there is a hyperfinite subfactor
$\mathcal{B}$ of $\mathcal{A}$ such that $\mathcal{B}' \cap
\mathcal{A} = \mathbb{C} I$. The following lemma  is essentially
Theorem 8 in \cite{AR4}.  The proof presented here is based on the
direct integral theory for von Neumann algebras and is different
from the one in \cite{AR4}.

\begin{lemma} (\cite{AR4})\label{3.12}
If $\mathcal{M}$ is a type II$_{1}$ von Neumann algebra acting on a
separable Hilbert space $H$, then there is a hyperfinite type
II$_{1}$ subfactor $\mathcal{R}$ of $\mathcal{M}$ such that
$\mathcal{R}' \cap \mathcal{M}=\mathcal{Z}$, where $\mathcal Z$ is the center of
$\mathcal{M}$.
\end{lemma}

\begin{proof}
By Lemma \ref{2.1}, $\mathcal{M}$ can be decomposed (relative to its
center) as a direct integral $\int_X \bigoplus \mathcal{M}_{s} d
\mu$ over a locally compact complete separable metric measure space
$(X,\mu)$ and $\mathcal{M}_{s}$ is a type II$_{1}$ factor almost
everywhere. In the following we assume that $\mathcal{M}_{s}$ is a
type II$_{1}$ factor for every $s \in X$.

By Remark \ref{2.4}, we can obtain a separable Hilbert space $K$ and
a family of unitaries $\{ U_{s}:H_{s} \to K; s \in X \}$ such that
the maps $s \to U_{s}x(s)$ and $s \to U_{s}a(s)U_{s}^{*}$ are
measurable for any $x \in H$ and any decomposable $a \in B(H)$. Let
$\mathcal{B}$ be the unit ball of $B(K)$ with the $*$-strong
operator topology. We observe that $\mathcal{B}$ is metrizable by
setting $d(S,T)=\sum_{j=1}^{\infty} 2^{-j}(\Vert (S-T)e_{j} \Vert
+\Vert (S^{*}-T^{*})e_{j} \Vert)$ for any $S, T \in
\mathcal{B}$, where $\{ e_{j}: j \in \mathbb{N}\}$ is an orthonormal
basis of $K$. Moreover, $(\mathcal{B},d)$ is a complete  separable
metric space. Let $\mathcal{C}=\mathcal{B} \times \mathcal{B}$ equipped with
the product topology. It follows that $\mathcal{C}$ is a complete separable
metric space.

Let $\{ a_{j}': j \in \mathbb{N} \}$ be a $SOT$ dense subset of the
unit ball $(\mathcal{M}')_{1}$. By Lemma 14.1.24, we may assume that
$(\mathcal{M}')_{s}=(\mathcal{M}_{s})'$ for every $s \in X$ and we
use the notation $\mathcal{M}'_{s}$ for both. By Remark \ref{3.6},
we may assume further that $\{ a'_{j}(s) : j \in \mathbb{N} \}$ is
$SOT$ dense in $(\mathcal{M}'_{s})_{1}$ for every $s \in X$. Let $\{
y_{j}: j \in \mathbb{N} \}$ be a countable dense subset in $H$. By
Lemma 14.1.3 in \cite{KR1}, the Hilbert space generated by $\{
y_{j}(s): j \in \mathbb{N} \}$ is $H_{s}$ for almost every $s \in
X$. Replacing $\{ y_{j}: j \in \mathbb{N} \}$ by the set of all
finite rational-linear combinations of vectors in $\{ y_{j}: j \in
\mathbb{N} \}$ if necessary, in the following we assume that $\{
y_{j}(s): j \in \mathbb{N} \}$ is dense in $H_{s}$ for every $s \in
X$.

Fix an irrational number $\theta \in (0,1)$.
We denote by $(s,W,V)$  an element in $ X\times \mathcal{B} \times \mathcal{B} =  X \times \mathcal{C}$.

The maps $W \to WW^{*}$, $W \to W^{*}W$, $V \to VV^{*}$, $V \to V^{*}V$ are $*$-$SOT$ continuous from $\mathcal{B}$ to $\mathcal{B}$. The maps $(W, V) \to WV, (W, V) \to e^{2\pi i \theta} VW$ are continuous from $\mathcal{C}$ with the product topology to $\mathcal{B}$ with the $*$-strong operator topology. Therefore the maps
\begin{align}
(s, W, V) &\to WW^{*}, \label{36} \\
(s, W, V) &\to W^{*}W, \label{37} \\
(s, W, V) &\to VV^{*}, \label{38} \\
(s, W, V) &\to V^{*}V, \label{39} \\
(s, W, V) &\to WV, \label{40} \\
(s, W, V) &\to e^{2\pi i \theta} VW \label{41}
\end{align}
are Borel measurable from $X \times \mathcal{C}$ to $\mathcal{B}$.
By Remark \ref{2.4}, the maps
$$s \to U_{s}a'_{j}(s)U_{s}^{*}$$
from $X$ to $B(K)$ and
$$s \to U_{s}y_{j}(s)$$
from $X$ to $K$ are all measurable for each $j \in \mathbb{N}$.

Let $$ \Bbb Q\langle X, Y, Z_1, Z_2,\ldots \rangle $$ be the collection of all $*$-polynomials   in intermediate variables $X, Y, Z_1, Z_2, \ldots$ with rational coefficients.
It is a countable set. By Lemma 14.3.1 in \cite{KR1}, there exists a Borel
$\mu$-null subset $N$ of $X$ such that,   $\forall j_{1}, j_{2} \in \mathbb{N}$, $\forall f\in \Bbb Q\langle X, Y, Z_1, Z_2,\ldots \rangle $, the maps
\begin{align}
(s, W ,V) &\to WU_{s}a'_{j}(s)U_{s}^{*}, \label{42} \\
(s, W, V) &\to U_{s}a'_{j}(s)U_{s}^{*}W, \label{43} \\
(s, W ,V) &\to VU_{s}a'_{j}(s)U_{s}^{*}, \label{44} \\
(s, W, V) &\to U_{s}a'_{j}(s)U_{s}^{*}V \label{45}
\end{align}
are Borel measurable from $(X \setminus N) \times \mathcal{C}$ to $\mathcal{B}$ and the map
\begin{eqnarray}
(s, W, V) \to \Vert f(W,V, \{ U_{s}a'_{j}(s)U_{s}^{*}: j \in \mathbb{N} \} )U_{s}y_{j_{1}}(s)-U_{s}y_{j_{2}(s)} \Vert \label{46}
\end{eqnarray}
is Borel meaurable from $(X \setminus N) \times \mathcal{C}$ to $\mathbb{C}$.

Now we introduce the  set $\eta$ as follows.

\vspace{0.2cm}

 {\em
Let $\eta$ be the collection of all these elements $ (s, W, V) \in (X \setminus N) \times \mathcal{C}$  satisfying
\begin{enumerate}
\item [(i)] $WW^{*}=W^{*}W=VV^{*}=V^{*}V=I$, where $I$ is the identity in $B(K)$;

\item [(ii)] $WU_{s}a'_{j}(s)U_{s}^{*}=U_{s}a'_{j}(s)U_{s}^{*}W$ and $VU_{s}a'_{j}(s)U_{s}^{*}=U_{s}a'_{j}(s)U_{s}^{*}V$ for every $j \in \mathbb{N}$;

\item [(iii)] $WV=e^{2\pi i \theta} VW$;

\item [(iv)] for all $N, j_{1}, j_{2} \in \mathbb{N}$, there exists an $f$  in $\Bbb Q\langle X, Y, Z_1, Z_2,\ldots \rangle $ such that

$$\Vert f(W,V, \{ U_{s}a'_{j}(s)U_{s}^{*}: j \in \mathbb{N} \} )U_{s}y_{j_{1}}(s)-U_{s}y_{j_{2}}(s) \Vert < 1/N.$$
\end{enumerate}}
\vspace{0.2cm}

{\bf Claim \ref{3.12}.1.} {\em The set $\eta$ is analytic.}

\vspace{0.2cm}

{\noindent  Proof of Claim \ref{3.12}.1:} Since the maps (\ref{36})-(\ref{46}) are all
Borel measurable when restricted to $(X\setminus N) \times
\mathcal{C}$, $\eta$ is a Borel set. By Theorem 14.3.5 in
\cite{KR1}, $\eta$ is analytic. This completes the proof of Claim \ref{3.12}.1.

\vspace{0.2cm}
{\bf Claim \ref{3.12}.2.} {\em Let $\pi$ be the projection of $X \times \mathcal{C}$ onto $X$. Then $\pi(\eta) = X \setminus N$.}

\vspace{0.2cm}

{\noindent  Proof of Claim \ref{3.12}.2:} We observe that an element $(s, W, V)$ satisfies conditions (i), (ii) and (iii) if and only if $U_{s}^{*}WU_{s}$ and $U_{s}^{*}VU_{s}$ are two unitaries in $\mathcal{M}_{s}$ such that
$(U_{s}^{*}WU_{s})(U_{s}^{*}VU_{s})=e^{2 \pi i \theta}(U_{s}^{*}VU_{s})(U_{s}^{*}WU_{s})$. Since $\{ y_{j}(s): j \in \mathbb{N} \}$ is dense in $H_{s}$ for every $s \in X$, condition (iv) is equivalent to the condition that the von Neumann algebra generated by $\{U_{s}^{*}WU_{s}, U_{s}^{*}VU_{s} \} \cup \{ a'_{j}(s): j \in \mathbb{N} \}$ is $B(H_{s})$.

For each $s \in X$, $\mathcal{M}_{s}$  is a type II$_{1}$ factor
with separable predual. By Popa's result in \cite{P}, there exists a
type II$_{1}$ hyperfinite subfactor $\mathcal{R}^{(s)}$ of
$\mathcal{M}_{s}$ such that $(\mathcal{R}^{(s)})' \cap
\mathcal{M}_{s}=\mathbb{C} I_{s}$. Notice a hyperfinite II$_1$ factor always contains an irrational rotation C$^*$-algebra as a SOT dense subalgebra. Combining with the argument in the prededing paragraph, we know that there
exist two unitaries $U_{s}^{*}WU_{s}$ and $U_{s}^{*}VU_{s}$ in
$\mathcal{R}^{(s)}$ (where $W, V$ are unitaries in $\mathcal{B}$)
such that they generate $\mathcal{R}^{(s)}$ as a von Neumann algebra
and  $(s,W,V)$ satisfies conditions (i), (ii) and (iii). The condition
$(\mathcal{R}^{(s)})' \cap \mathcal{M}_{s}=\mathbb{C} I_{s}$ is
equivalent to the condition that the von Neumann algebra generated
by $\mathcal{R}^{(s)} \cup \mathcal{M}'_{s}$ is $B(H_{s})$. Since
$U_{s}^{*}WU_{s}$ and $U_{s}^{*}VU_{s}$ generate
$\mathcal{R}^{(s)}$ as a von Neumann algebra and $\{ a'_{j}(s): j
\in \mathbb{N} \}$ is $SOT$ dense in the unit ball of
$\mathcal{M}'_{s}$, the von Neumann algebra
$W^{*}(U_{s}^{*}WU_{s}, U_{s}^{*}VU_{s}, \{ a'_{j}(s): j
\in \mathbb{N} \})$ generated by $U_{s}^{*}WU_{s},
U_{s}^{*}VU_{s}$ and $\{ a'_{j}(s): j \in \mathbb{N} \}$ is
$B(H_{s})$. Hence,  from the argument in the preceding paragraph, it follows that  $(s,W,V)$ satisfies  satisfy condition (iv). Therefore  the
image of $\eta$ under $\pi$ is $X \setminus N$. This completes the proof of  Claim \ref{3.12}.2.

\vspace{0.2cm}

\noindent(Continue the proof of Lemma \ref{3.12}) By Claim
\ref{3.12}.1 and Claim \ref{3.12}.2, we know that $\eta$ is analytic
and $\pi(\eta)=X \setminus N$. By the measure-selection principle
(Theorem 14.3.6 in \cite{KR1}), there is a measurable map $s \to
(W_{s}, V_{s})$ from $X \setminus N$ to $\mathcal{C}$ such that $(s,
W_{s}, V_{s})$ satisfies condition (i), (ii), (iii) and (iv) for
$s \in X \setminus N$ almost everywhere. Defining $W_{s}=V_{s}=0$
for any $s \in N$, we get {\em a measurable map \begin{equation} s
\to (W_{s}, V_{s})\label {3.12.1}\end{equation} from $X$ to
$\mathcal{C}$  such that $(s, W_{s}, V_{s})$ satisfies condition
(i), (ii), (iii) and (iv) for   $s \in X $ almost everywhere.}

 For any two vectors $x, y \in H$, we have
\begin{equation} \langle U_{s}^{*}W_{s}U_{s}x(s), y(s) \rangle=\langle W_{s}U_{s}x(s), U_{s}y(s) \rangle.\label {3.12.2}\end{equation}
Combining (\ref{3.12.2}) with  (\ref{3.12.1}), we know the map $$s \to \langle U_{s}^{*}W_{s}U_{s}x(s), y(s)  \rangle$$ from $X$ to $\mathbb{C}$
 is measurable. Since
$$\vert \langle U_{s}^{*}W_{s}U_{s}x(s), y(s)  \rangle \vert \leq \Vert x(s) \Vert \Vert y(s) \Vert,$$
we obtain that $$s \to \langle U_{s}^{*}W_{s}U_{s}x(s), y(s)  \rangle$$ is
integrable. By Definition 14.1.1 in \cite{KR1}, it follows that
$$U_{s}^{*}W_{s}U_{s}x(s)=(\bar{W} x)(s)$$
almost everywhere for some $\bar{W} x \in H$. Therefore
\begin{eqnarray}
\bar{W}(s)=U_{s}^{*}W_{s}U_{s} \label{47}
\end{eqnarray}
for almost every $s \in X$. Since conditions (i) and (ii) imply that
$U_{s}^{*}W_{s}U_{s}$ is a unitary in $\mathcal M_s$, we obtain from
equation (\ref{47}) that $\bar{W}$ is a unitary in $\mathcal{M}$.
Similarly we can find another unitary $\bar{V}$ in $\mathcal{M}$
such that
\begin{eqnarray}\bar{V} (s)=U_{s}^{*}V_{s}U_{s} \label{47.2}
\end{eqnarray} for almost every $s \in X$ and
thus, from condition (iii),
$$\bar{W}(s) \bar{V}(s)=e^{2 \pi i \theta} \bar{V}(s) \bar{W}(s)$$
for almost every $s \in X$. Therefore \begin{eqnarray}\bar{W}
\bar{V}=e^{2 \pi i \theta} \bar{V} \bar{W}.\label{47.1}
\end{eqnarray}

Let $\mathcal{R}^{(s)}$ be the von Neumann subalgebra generated by
$U_{s}^{*}W_{s}U_{s}$ and $U_{s}^{*}V_{s}U_{s} $ in $\mathcal M_s$.
From condition (iv), we know that $(\mathcal{R}^{(s)})'\cap M_s=\Bbb
CI_s$ for $s\in X$ almost everywhere.

Let $\mathcal{R}$ be a von Neumann subalgebra of $\mathcal{M}$
generated by   two unitaries $\bar{W}, \bar{V}$. From (\ref{47}),
(\ref{47.2}) and (\ref{47.1}), it follows that $\mathcal{R}$ is a
hyperfinite type II$_{1}$ factor and
$\mathcal{R}_{s}=\mathcal{R}^{(s)}$ for almost every $s \in X$.

To complete the proof, we just need to show that $\mathcal{R}' \cap \mathcal{M} = \mathcal{Z}$. Suppose $a \in \mathcal{R}' \cap \mathcal{M}$.  Then $a(s) \in \mathcal{R}'_{s} \cap \mathcal{M}_{s}$ for almost every $s \in X$. Since $(\mathcal{R}^{(s)})' \cap \mathcal{M}_{s}=\mathbb{C} I_{s}$ and $\mathcal{R}_{s}=\mathcal{R}^{(s)}$ for almost every $s \in X$, $a(s)=c_{s}I_{s}$ for almost every $s \in X$ and thus $a \in \mathcal{Z}$. Hence $\mathcal{R}' \cap \mathcal{M} = \mathcal{Z}$.
\end{proof}

 The following result is a   generalization of Theorem 5.4 in \cite{EFAR2} in the setting of von Neumman algebras.
 The proof follows the similar line as  the one used in  Theorem 5.4 in \cite{EFAR2}.

\begin{theorem} \label{3.13}
Let $\mathcal{M}$ be a type II$_{1}$ von Neumann algebra with
separable predual and $\mathcal{Z}$ be the center of $\mathcal{M}$.
Let $\rho$ be a faithful normal tracial state on $\mathcal{M}$ and
$\Vert \cdot \Vert_{2}$ be the $2$-norm on $\mathcal{M}$ induced by
$\rho$. If $\mathcal{M}$ has Property $\Gamma$, then there exists a
hyperfinite type II$_{1}$ subfactor $\mathcal{R}$ of $\mathcal{M}$
such that
\begin{enumerate}
\item [(I)] $\mathcal{R} \cap \mathcal{M}' = \mathcal{Z}$;
\item [(II)] for any $n \in \mathbb{N}$, any elements $a_{1}, a_{2},
..., a_{k}$ in $\mathcal{M}$, there exists a countable collection of
projections $\{ p_{1t}, p_{2t}, \dots, p_{nt} : t \in \mathbb{N} \}$
in $\mathcal{R}$ such that
\begin{enumerate}
\item [(i)] for each $t \in \mathbb{N}$, $p_{1t}, p_{2t}, \dots,
p_{nt}$ are $n$ orthogonal equivalent projections in $\mathcal{R}$
with sum I;

\item [(ii)] $\lim\limits_{t \to \infty} \Vert p_{it}a_{j}-a_{j}p_{it}
\Vert_{2} =0$ for any $i=1, 2,\dots , n;  j=1, 2, \dots, k$.
\end{enumerate}
\end{enumerate}
\end{theorem}
\begin{proof}
Since $\mathcal{M}$ has separable predual, by Proposition A.2.1 in
\cite{JS}, there is a faithful normal representation $\pi$ of
$\mathcal{M}$ on a separable Hilbert space. Replacing $\mathcal{M}$
by $\pi({\mathcal{M}})$  and $\rho$ by $\rho \circ \pi^{-1}$, we may
assume that $\mathcal{M}$ is acting on a separable Hilbert space
$H$.

By Lemma \ref{2.1}, there are direct integral decompositions
$\mathcal{M}=\int_{X} \bigoplus \mathcal{M}_{s} d \mu $ and
$H=\int_{X} H_{s} d \mu$ of $(\mathcal{M}, H)$ relative to
$\mathcal{Z}$ over $(X, \mu)$, where $\mathcal{M}_{s}$ is a type
II$_{1}$ factor for almost every $s \in X$. We assume that every
$\mathcal{M}_{s}$ is a type II$_{1}$ factor. Notice that $\rho$ is a faithful, normal, tracial state on $\mathcal M$. From Lemma \ref{2.2}, we might assume there is a positive, faithful, normal, tracial linear functional $ \rho_s $ on $\mathcal M_s$ for every $s\in X$  such that
$$
\rho(a)=\int_X \rho_s(a(s)) d\mu, \qquad \forall a\in \mathcal M.
$$

Let $\{\phi_{i}: i \in \mathbb{N} \}$ be a sequence of normal states on $\mathcal{M}$ that is norm dense in the set of all
  normal states on $\mathcal{M}$. Let $\{ b_{j}: j \in \mathbb{N} \}$ be a sequence of elements that is $SOT$ dense in the unit ball
  $(\mathcal{M})_{1}$ of $\mathcal{M}$. By Remark \ref{3.6}, we may assume that $\{b_{j}(s): j \in \mathbb{N} \}$ is $SOT$ dense in the
  unit ball $(\mathcal{M}_{s})_{1}$ of $\mathcal{M}_{s}$ for every $s \in X$.

   Let $\tau$ be the unique center-valued trace on $\mathcal{M}$ such that $\tau(a)=a$ for all $a\in \mathcal Z$ (see Theorem 8.2.8 in \cite{KR1}).

We will show that {\em there is an increasing sequence $\{
\mathcal{A}_{t}: t \in \mathbb{N} \}$  of full matricial algebras in
$\mathcal M$ satisfying, for all $t\in \Bbb N$,
\begin{enumerate}
\item [(a)] there exists a $\mu$-null subset $N_{t}$ of $X$ such that, for each $1 \leq l \leq t$, there exist $l$ equivalent orthogonal projections $p_{1}, p_{2}, \dots, p_{l}$ in $\mathcal{A}_{t}$ with sum $I$ satisfying

$$\rho_{s}((p_{i}(s)b_{j}(s)-b_{j}(s)p_{i}(s))^{*}(p_{i}(s)b_{j}(s)-b_{j}(s)p_{i}(s)))<1/t$$
 for any $i=1, 2, \dots, l;  j=1,2, \dots , t;  s \in X \setminus N_{t};$
\item [(b)] let $\mathcal{U}_{t}$ be the unitary group of $\mathcal{A}_{t}$ and $d \mu_{t}$ be the normalized Haar measure on $\mathcal{U}_{t}$. Then for any $i, j =1, 2, \dots, t$,

$$\vert \phi_{i}(\int_{\mathcal{U}_{t}} ub_{j}u^{*} d\mu_{t} -\tau(b_{j})) \vert < 1/t.$$\end{enumerate}}

First, we observe that conditions (a) and (b) are
satisfied by letting $\mathcal{A}_{1}=\mathbb{C}1$ and $p_{1}=I$.
Now suppose $\mathcal{A}_{t-1}$ has been constructed. Take
$\mathcal{N}_{1}=\mathcal{A}_{t-1}' \cap \mathcal{M}$. By Lemma
11.4.11 in \cite{KR1}, $\mathcal{M} \cong \mathcal{A}_{t-1}
{\otimes} \mathcal{N}_{1}$.

Next, in order to construct $\mathcal A_t$, we will apply Lemma
\ref{3.10} $t-1$ times. At the first time, applying Lemma \ref{3.10}
to $\mathcal A_{t-1}$ and the set $\{ b_{1}, b_{2}, \dots, b_{t} \}$, we
obtain two equivalent orthogonal projections $p_{1,1}, p_{2,1}$ in
$\mathcal{N}_{1}$ with sum $I$ and a $\mu$-null subset $N_{t,1}$ of
$X$ such that
\begin{eqnarray}
\rho_{s}((p_{i,1}(s)b_{j}(s)-b_{j}(s)p_{i,1}(s))^{*}(p_{i,1}(s)b_{j}(s)-b_{j}(s)p_{i,1}(s)))<1/t \label{48}
\end{eqnarray}
for any $i=1, 2; j=1, 2, \dots, t, s \in X \setminus N_{t,1}$. Note
that  $p_{1,1}, p_{2,1}$ are two equivalent orthogonal projections
 in $\mathcal{N}_{1}$ with sum $I$. There is a unital subalgebra
$\mathcal{B}_{t,1}$ of $\mathcal{N}_{1}$ such that
$\mathcal{B}_{t,1} \cong M_{2}(\mathbb{C})$ and $p_{1,1}, p_{2,1}
\in \mathcal{B}_{t,1}$. Take
\begin{eqnarray}
\mathcal{A}_{t,1}=\mathcal{A}_{t-1} {\otimes} \mathcal{B}_{t,1}.
\label{49}
\end{eqnarray}
Now suppose that $\mathcal{A}_{t,l-1}$ have been constructed for
some $2\le l \leq t-1$. By applying Lemma \ref{3.10} to
$\mathcal{A}_{t,l-1}$ and $\{ b_{1}, b_{2}, \dots, b_{t} \}$, we can
find $l+1$ equivalent  orthogonal projections $p_{1,l}, p_{2,l},
\dots, p_{l+1, l}$ in $\mathcal{A}_{t, l-1}' \cap \mathcal{M}$ with
sum $I$ and a $\mu$-null subset $N_{t, l}$ of $X$ such that
\begin{eqnarray}
\rho_{s}((p_{i,l}(s)b_{j}(s)-b_{j}(s)p_{i,l}(s))^{*}(p_{i,l}(s)b_{j}(s)-b_{j}(s)p_{i,l}(s)))<1/t \label{50}
\end{eqnarray}
for any $i=1, 2, \dots, l+1, j=1, 2, \dots, t,$ and $ s \in X
\setminus N_{t, l}$. Again there is a unital subalgebra
$\mathcal{B}_{t, l}$ of $\mathcal{A}_{t, l-1}' \cap \mathcal{M}$
such that $\mathcal{B}_{t, l} \cong M_{l+1}(\mathbb{C})$ and
$p_{1,l}, p_{2,l}, \dots, p_{l+1,l} \in \mathcal{B}_{t, l}$. Take
\begin{eqnarray}
\mathcal{A}_{t, l}=\mathcal{A}_{t, l-1} {\otimes} \mathcal{B}_{t,
l}. \label{51}
\end{eqnarray}
Now we let
\begin{eqnarray}
\mathcal{B}_{t}=\mathcal{A}_{t, t-1} \label{52}
\end{eqnarray}
and
\begin{eqnarray}
N_{t}=\cup_{l=1}^{t-1} N_{t, l}. \label{53}
\end{eqnarray}
Then $\mu(N_{t})=0$. By (\ref{48}), (\ref{49}), (\ref{50}), (\ref{51}), (\ref{52}) and (\ref{53}),  $\mathcal{B}_{t}$ contains   sets  of projections satisfying condition (a).

Let $\mathcal{N}=\mathcal{B}_{t}' \cap \mathcal{M}$. By Lemma
11.4.11 in \cite{KR1}, we know that $\mathcal{M} \cong
\mathcal{B}_{t} {\otimes} \mathcal{N}$. By the arguments in Section
11.2 in \cite{KR1}, $\mathcal{N}$ is a type II$_{1}$ von Neumann
algebra, and therefore, by Lemma \ref{3.12}, there is a hyperfinite
subfactor $\mathcal{S}$ of $\mathcal{N}$ such that $\mathcal{S}'
\cap \mathcal{N}=\mathcal{Z}_{\mathcal{N}}$, where
$\mathcal{Z}_{\mathcal{N}}$ is  the center of $\mathcal{N}$. Hence
$(\mathcal{B}_{t} {\otimes} \mathcal{S})' \cap
\mathcal{M}=\mathbb{C}I {\otimes}
\mathcal{Z}_{\mathcal{N}}=\mathcal{Z}$. Since $\mathcal{S}$ is a
hyperfinite type II$_{1}$ factor, there exists an increasing
sequence $\{ \mathcal{F}_{r}: r \in \mathbb{N} \}$ of matrix
subalgebras of $\mathcal{S}$ whose union is ultraweakly dense in
$\mathcal{S}$ and thus $\cup_{r \in \mathbb{N}} \mathcal{B}_{t}
{\otimes} \mathcal{F}_{r}$ is ultraweakly dense in $\mathcal{B}_{t}
{\otimes} \mathcal{S}$. Let $\mathcal{V}_{r}$ be the unitary group
of $\mathcal{B}_{t} {\otimes} \mathcal{F}_{r}$ with normalized Haar
measure $d \nu_{r}$. Since $(\mathcal{B}_{t} {\otimes} \mathcal{S})'
\cap \mathcal{M} = \mathcal{Z}$ and $\tau$ is a center-valued trace
on $\mathcal M$ such that $\tau(a)=a$ for all $a\in \mathcal Z$,
Lemma 5.4.4 in \cite{AR3} shows that $\tau(a)= \lim\limits_{r \to
\infty} \int_{\mathcal{V}_{r}} vav^{*} d \nu_{r}$ ultraweakly for
all $a \in \mathcal{M}$. Since each $\phi_{i}$ is normal, there
exists $r$ large enough such that
\begin{eqnarray}
\vert \phi_{i}(\int_{\mathcal{V}_{r}}vb_{j}v^{*}d \nu_{r}-\tau (b_{j})) \vert < 1/t,   \forall i, j=1, 2, \dots, t. \label{54}
\end{eqnarray}
Now we let $$\mathcal{A}_{t}=\mathcal{B}_{t} {\otimes} \mathcal{F}_{r}.$$ Then $\mathcal{A}_{t}$
satisfies both conditions (a) and (b). The construction is finished.

Let $\mathcal{R} \subset \mathcal{M}$ be the ultraweak closure of
$\cup_{t \in \mathbb{N} } \mathcal{A}_{t}$. It follows that
$\mathcal{R}$ is a finite von Neumann algebra containg an
ultraweakly dense matricial C$^*$-algebra. By Corollary 12.1.3
in \cite{KR1}, $\mathcal{R}$ is a hyperfinite type II$_{1}$
subfactor of $\mathcal{M}$.

Now fix $n \in \mathbb{N}$, $\epsilon >0$ and elements $a_{1}, a_{2}, \dots, a_{k}$ in $\mathcal{M}$. We may first assume that $\Vert a_{l} \Vert \leq 1$ for any $1 \leq l \leq k$. Since $\{b_{j}: j \in \mathbb{N} \}$ is $SOT$ dense in the unit ball of $\mathcal{M}$, there exist elements $b_{j_{1}}, b_{j_{2}}, \dots, b_{j_{k}}$ such that
\begin{eqnarray}
\Vert a_{l}-b_{j_{l}} \Vert_{2} < \epsilon /3 \label{55}
\end{eqnarray}
for any $1 \leq l \leq k$.  For each integer $t > \max \{ n,  j_{1},  j_{2}, \dots, j_{k} \}$, by condition (a), there exist a $\mu$-null subset $N_{t}$ of $X$ and a set of $n$ orthogonal equivalent projections $\{ p_{1t}, p_{2t}, \dots, p_{nt}\}$ in $\mathcal{A}_{t}$ such that
\begin{eqnarray}
\rho_{s} ((p_{it}(s)b_{j_{l}}(s)-b_{j_{l}}(s)p_{it}(s))^{*}(p_{it}(s)b_{j_{l}}(s)-b_{j_{l}}(s)p_{it}(s)) < 1/t \label{56}
\end{eqnarray}
for  all $i\in \{1, 2, \dots, n\}$, $  l\in \{1,2, \dots , k\}$, and
$s \in X \setminus N_{t}.$

Take $N=\cup_{t \in \mathbb{N}} N_{t}$. Then $\mu (N)=0$ and inequality (\ref{56}) implies
\begin{eqnarray}
\lim\limits_{t \to \infty} \rho_{s} ((p_{it}(s)b_{j_{l}}(s)-b_{j_{l}}(s)p_{it}(s))^{*}(p_{it}(s)b_{j_{l}}(s)-b_{j_{l}}(s)p_{it}(s)) = 0 \label{57}
\end{eqnarray}
for all $i\in \{1, 2, \dots ,n\},$ $l\in \{1, 2, \dots, k\}$,  and $
s \in X \setminus N$. For any fixed $i \in \{ 1, 2, \dots ,n \},  l
\in \{ 1, 2, \dots, k \}$, define function $f_{t}: X \to \mathbb{C}$
such that
$$f_{t}(s)=\rho_{s} ((p_{it}(s)b_{j_{l}}(s)-b_{j_{l}}(s)p_{it}(s))^{*}(p_{it}(s)b_{j_{l}}(s)-b_{j_{l}}(s)p_{it}(s)).$$
Then $\vert f_{t}(s) \vert \leq \rho_{s}(4I_{s})$ for almost every $s \in X$. Since
$$\int_{X} \rho_{s}(4I_{s}) d \mu = \rho (4I)=4,$$
by the Dominated Convergence Theorem, (\ref{57}) gives
$$\lim\limits_{t \to \infty} \rho ((p_{it}b_{j_{l}}-b_{j_{l}}p_{it})^{*}(p_{it}b_{j_{l}}-b_{j_{l}}p_{it})) = 0$$
for  all $i\in \{1, 2, \dots, n\}$, $  l\in \{1,2, \dots , k\}$.
Hence there exists $t_{0} \in \mathbb{N}$ such that
\begin{eqnarray}
\Vert p_{it}b_{j_{l}}-b_{j_{l}}p_{it} \Vert_{2} < \epsilon /3 \label{58}
\end{eqnarray}
for all $i\in \{1, 2, \dots, n\}$, $  l\in \{1,2, \dots , k\}$ and $
t>t_{0}$.

Therefore for any $t>t_{0}$, it follows from (\ref{55}) and (\ref{58}) that
\begin{eqnarray*}
\Vert p_{it}a_{l}-a_{l}p_{it} \Vert_{2}
&\leq& \Vert p_{it}b_{j_{l}}-b_{j_{l}}p_{it} \Vert_{2} + \Vert p_{it}(a_{l}-b_{j_{l}})-(a_{l}-b_{j_{l}})p_{it} \Vert_{2} \\
&\leq& \Vert p_{it}b_{j_{l}}-b_{j_{l}}p_{it} \Vert_{2} + 2 \Vert a_{l}-b_{j_{l}} \Vert_{2}\\
&<& \epsilon
\end{eqnarray*}
for all $i\in \{1, 2, \dots, n\}$, $  l\in \{1,2, \dots , k\}$.
Hence $\lim\limits_{t \to \infty} \Vert p_{it}a_{l}-a_{l}p_{it}
\Vert_{2} = 0$ for  all $i\in \{1, 2, \dots, n\}$, $  l\in \{1,2,
\dots , k\}$.

It remains to show that $\mathcal{R}' \cap \mathcal{M}=\mathcal{Z}$.
Suppose that $a \in \mathcal{R}' \cap \mathcal{M}$ and $\Vert a
\Vert =1$. Since the sequence $\{ b_{j}: j \in \mathbb{N} \}$ is
$SOT$ dense in the unit ball of $\mathcal{M}$, we can choose a
subsequence $\{ b_{j_{l}}: l \in \mathbb{N} \}$ that converges to
$a$ in the strong operator topology. Therefore this subsequence
converges to $a$ ultraweakly. By the fact that $\tau$ is ultraweakly
continuous, $\lim\limits_{l \to \infty} \tau (b_{j_{l}}) = \tau (a)$
ultraweakly.  Since $a \in \mathcal{R}'$, for each $i \in
\mathbb{N}$,
\begin{eqnarray*}
\vert \phi_{i}(\int_{\mathcal{U}_{j_{l}}} ub_{j_{l}}u^{*} d \mu_{j_{l}}-a) \vert &=& \vert \phi_{i}(\int_{\mathcal{U}_{j_{l}}} u(b_{j_{l}}-a)u^{*} d \mu_{j_{l}})  \vert\\
&\leq& (\phi_{i}((b_{j_{l}}-a)^{*}(b_{j_{l}}-a)))^{1/2} \\
&\to& 0.\\
\end{eqnarray*}
From the fact that the sequence $\{ \phi_{i} : i \in \mathbb{N} \}$
is norm dense in the set of normal states on $\mathcal{M}$, we get
that $\int_{\mathcal{U}_{j_{l}}} ub_{j_{l}}u^{*} d \mu_{j_{l}}$
converges to $a$ ultraweakly. By condition (b),
$\int_{\mathcal{U}_{j_{l}}} ub_{j_{l}}u^{*} d \mu_{j_{l}}$ converges
to $\tau(a)$ ultraweakly. Therefore $a=\tau (a)$ and thus $a \in
\mathcal{Z}$. Hence $\mathcal{R}' \cap \mathcal{M} = \mathcal{Z}$.
The proof is complete.
\end{proof}

\section{Necessary inequalities}
Suppose $\mathcal{M}$ is a von Neumann algebra and $\mathcal{N}$ is
a von Neumann subalgebra of $\mathcal{M}$. A map $\phi:
\mathcal{M}^{k} \to B(H)$ is called $\mathcal{N}$-multimodular if,
 for any $s \in \mathcal{N}$ and any $a_{1}, a_{2}, \dots, a_{k} \in
\mathcal{M}$,
$$s\phi (a_{1}, a_{2}, \dots , a_{k})=\phi (sa_{1}, a_{2}, \dots , a_{k}),$$
$$\phi (a_{1}, a_{2}, \dots , a_{k})s=\phi (a_{1}, a_{2}, \dots , a_{k}s),$$
$$\phi(a_{1}, a_{2}, \dots , a_{i}s, a_{i+1}, \dots , a_{k})=\phi (a_{1}, a_{2}, \dots , a_{i}, sa_{i+1}, \dots , a_{k}).$$
For any $n \in \mathbb{N}$, the $n$-fold amplication  $\phi^{(n)} :
(M_{n}(\mathcal{M}))^{k} \to M_{n}(\mathcal{M})$ of a bounded map
$\phi : \mathcal{M}^{k} \to \mathcal{M}$ is defined in \cite{EA1}
and \cite{EA2} as follows: for elements $(a_{ij}^{(1)}),
(a_{ij}^{(2)}), \dots , (a_{ij}^{(k)})$ in $M_{n}(\mathcal{M})$, the
$(i,j)$ entry of $\phi^{(n)} ((a_{ij}^{(1)}), (a_{ij}^{(2)}), \dots
, (a_{ij}^{(k)}))$ is
$$\sum\limits_{1 \leq j_{1}, j_{2}, \dots , j_{n-1} \leq n} \phi (a_{ij_{1}}^{(1)}, a_{j_{1}j_{2}}^{(2)}, \dots , a_{j_{n-2}j_{n-1}}^{(n-1)}, a_{j_{n-1}j}^{(n)}).$$
A bounded map $\phi$ is said to be completely bounded if $\sup\limits_{n \in \mathbb{N}}  \{ \Vert \phi^{(n)} \Vert : n \in \mathbb{N} \} < \infty$. When $\phi$ is completely bounded, we denote $\Vert \phi \Vert_{cb} =\sup\limits_{n \in \mathbb{N}}  \{ \Vert \phi^{(n)} \Vert : n \in \mathbb{N} \}$.

Let $\{ e_{ij} \}_{i,j=1}^{n}$ be the standard matrix units for $M_{n}(\mathbb{C})$. Then
$$\Vert \phi^{(n)} (e_{11}a_{1}e_{11}, e_{11}a_{2}e_{11}, \dots , e_{11}a_{k}e_{11}) \Vert \leq \Vert \phi \Vert \Vert a_{1} \Vert \dots \Vert a_{k} \Vert$$
for any $a_{1}, a_{2}, \dots , a_{k}$ in $M_{n}(\mathcal{M})$.

If $\mathcal{M}$ is a type II$_{1}$ von Neumann algebra and $n$ is a
positive integer, $M_{n}(\mathcal{M})$ is also a type II$_{1}$ von
Neumann algebra. In the rest of this section, we let $\tau_{n}$ be
the center-valued trace on $M_{n}(\mathcal{M})$ such that $\tau_{n}(a)=a$ for any $a$ in the center of $M_{n}(\mathcal{M})$ (see Theorem 8.2.8 in \cite{KR1}).
Let
\begin{eqnarray}
\gamma_{n}(a)=(\Vert a \Vert ^{2} + n \Vert \tau_{n} (a^{*} a) \Vert )^{1/2} \label{59}
\end{eqnarray}
for each $a \in M_{n}(\mathcal{M})$.

\vspace{0.2cm}
Replacing $tr_{n}$ by $\tau_{n}$ in the proof of Lemma 3.1 in
\cite{EFAR2}, we can obtain the next lemma directly.
\begin{lemma} \label{4.1}
Let $\mathcal{M}$ be a type II$_{1}$ von Neumann algebra acting on a
Hilbert space $H$. Suppose $\mathcal{R}$ is a hyperfinite type
II$_{1}$ subfactor of $\mathcal{M}$ such that $\mathcal{R}' \cap
\mathcal{M} = \mathcal{Z}$, the center of $\mathcal{M}$. Let
$\theta$ be a positive number and $n$ be a positive integer. If
$\psi: M_{n}(\mathcal{M}) \times M_{n}(\mathcal{M}) \to B(H^{n})$ is
a normal bilinear map satisfying
$$\psi (ac, b)=\psi (a, cb), a, b \in M_{n}(\mathcal{M}), c \in M_{n}(\mathcal{R})$$
and
$$\Vert \psi( ae_{11}, e_{11}b ) \Vert \leq \theta \Vert a \Vert \Vert b \Vert,   a, b \in M_{n}(\mathcal{M}),$$
then
$$\Vert \psi(a, b) \Vert \leq \theta \gamma_{n}(a) \gamma_{n}(b)$$
for any $a, b \in M_{n}(\mathcal{M})$
\end{lemma}

If  Lemma 3.1 in \cite{EFAR2} is replaced by the preceding Lemma \ref{4.1}, the proof of Theorem 3.3 in \cite{EFAR2} gives us the following result.

\begin{lemma} \label{4.2}
Let $\mathcal{M}$ be a type II$_{1}$ von Neumann algebra acting on a
Hilbert space $H$. Suppose $\mathcal{M}$ has a hyperfinite subfactor
$\mathcal{R}$ such that $\mathcal{R}' \cap \mathcal{M} =
\mathcal{Z}$, the center of $\mathcal{M}$. Fix $k \in \mathbb{N}$.
If $\phi : \mathcal{M}^{k} \to B(H)$ is a $k$-linear
$\mathcal{N}$-multimodular normal map, then

$$\Vert \phi^{(n)} (a_{1}, a_{2}, \dots, a_{k}) \Vert \leq 2^{k/2} \Vert \phi \Vert \gamma_{n}(a_{1}) \gamma_{n}(a_{2}) \dots \gamma_{n}(a_{k})$$\\
for all $a_{1}, a_{2}, \dots, a_{k} \in M_{n}(\mathcal{M})$ and $n \in \mathbb{N}$.
\end{lemma}
\begin{corollary} \label{4.3}
Let $\mathcal{M}$ be a type II$_{1}$ von Neumann algebra and
$\mathcal{Z}$ the center of $\mathcal{M}$. Suppose $\mathcal{R}$
is a hyperfinite type II$_{1}$ subfactor of $\mathcal{M}$ such that
$\mathcal{R}' \cap \mathcal{M}=\mathcal{Z}$, the center of
$\mathcal{M}$. Let $n ,k \in \mathbb{N}$. Suppose $p_{1}, p_{2},
\dots , p_{n}$ are $n$ orthogonal equivalent projections in
$M_{n}(\mathcal{M})$ with sum $I$ and $\phi : \mathcal{M}^{k} \to
B(H)$ is a $k$-linear $\mathcal{R}$-multimodular normal map. Then
$$\Vert \phi^{(n)}(a_{1}p_{j}, a_{2}p_{j}, \dots , a_{k}p_{j}) \Vert \leq 2^{k} \Vert \phi \Vert \Vert a_{1} \Vert \Vert a_{2} \Vert \dots \Vert a_{k} \Vert $$
for any $j=1, 2, \dots, n$ and any $a_{1}, a_{2}, \dots , a_{k} \in M_{n}(\mathcal{M})$.
\end{corollary}

\begin{proof}
By Lemma \ref{4.2}, for any $j=1, 2, \dots, n$,
\begin{eqnarray}
\Vert \phi^{(n)} (a_{1}p_{j}, a_{2}p_{j}, \dots , a_{k}p_{j}) \Vert \leq 2^{k/2} \Vert \phi \Vert \gamma_{n}(a_{1}p_{j}) \gamma_{n}(a_{2}p_{j}) \dots \gamma_{n}(a_{k}p_{j}). \label{73}
\end{eqnarray}

Since $p_{1}, p_{2}, \dots , p_{n}$ are orthogonal equivalent projections with sum $I$, $\tau_{n}(p_{j})=\frac{1}{n} I$ for each $j$. Then for any $1 \leq i \leq k$,
\begin{eqnarray*}
\gamma_{n}(a_{i}p_{j})&=&(\Vert a_{i}p_{j} \Vert^{2} + n \Vert \tau_{n}(p_{j} a_{i}^{*}a_{i}p_{j}) \Vert )^{1/2}\\
&\leq& (\Vert a_{i} \Vert)^{2} +n \Vert a_{i} \Vert ^{2} \Vert \tau_{n}(p_{j}) \Vert)^{1/2}\\
&=&\sqrt{2} \Vert a_{i} \Vert.\\
\end{eqnarray*}
Therefore (\ref{73}) gives
$$\Vert \phi^{(n)}(a_{1}p_{j}, a_{2}p_{j}, \dots , a_{k}p_{j}) \Vert \leq 2^{k} \Vert \phi \Vert \Vert a_{1} \Vert \Vert a_{2} \Vert \dots \Vert a_{k} \Vert $$
for any $j=1, 2, \dots, n$ and any $a_{1}, a_{2}, \dots , a_{k} \in M_{n}(\mathcal{M})$.
\end{proof}

\section{Hochschild cohomology of type II$_{1}$ von Neumann algebras with separable predual and Property $\Gamma$}
Let us recall some notations from \cite {EFAR2}. Let $S_{k}$, $k
\geq 2$, be the set of nonempty subsets of $\{ 1, 2, \dots , k \}$.
Suppose $\phi: \mathcal{M}^{k} \to B(H)$ is a $k$-linear map, $p$ is
a projection in $\mathcal{M}$ and $\sigma \in S_{k}$.

  Define $\phi_{\sigma, p} : \mathcal{M}^{k} \to B(H)$ by
$$\phi_{\sigma, p}(a_{1}, \dots , a_{k}) = \phi (b_{1}, b_{2}, \dots ,
b_{k}),$$
where $b_{i}= pa_{i}-a_{i}p$ for $i \in \sigma$ and $b_{i}=a_{i}$
otherwise.

  Denote by $l(\sigma)$ the least integer in $\sigma$. Define $\phi_{\sigma, p, i}: \mathcal{M}^{k} \to B(H)$ by changing the $i$-th variable in $\phi_{\sigma, p}$ from $a_{i}$ to $pa_{i}-a_{i}p, 1 \leq i < l(\sigma) $, and replacing $pa_{i}-a_{i}p$ by $p(pa_{i}-a_{i}p)$ if $i=l(\sigma)$.

The following is Lemma 6.1 in \cite{EFAR2}.

\begin{lemma} \label{5.1}(\cite{EFAR2})
Let $p$ be a projection in a von Neumann algebra $\mathcal{M}$. Let $\mathcal{C}_{k}, k\geq 2$, be the set of $k$-linear maps $\phi : \mathcal{M}^{k} \to B(H)$ satisfying
\begin{equation}p\phi (a_{1}, a_{2}, \dots, a_{k})=\phi (pa_{1}, a_{2}, \dots, a_{k})\label{equa 5.1.1}\end{equation}
and
\begin{equation}\phi (a_{1}, \dots, a_{i}p, a_{i+1}, \dots, a_{k})= \phi (a_{1}, \dots, a_{i}, pa_{i+1}, \dots, a_{k})\label{equa 5.1.2}\end{equation}
for any $a_{1}, a_{2}, \dots, a_{k} \in \mathcal{M}$ and $1 \leq i \leq k-1$. Then if $\phi \in \mathcal{C}_{k}$,
$$p\phi (a_{1}, a_{2}, \dots, a_{k})-p\phi (a_{1}p, \dots, a_{k}p)=\sum\limits_{\sigma \in S_{k}}(-1)^{\vert \sigma \vert +1} p\phi_{\sigma, p}(a_{1}, \dots, a_{k}).$$
Moreover, for each $\sigma \in S_{k}$,
$$p\phi_{\sigma, p}(a_{1}, \dots, a_{k})=\sum\limits_{i=1}^{l(\sigma)} \phi_{\sigma, p, i}(a_{1}, a_{2}, \dots, a_{k}).$$
\end{lemma}

Let $\mathcal{M}$ be a type II$_{1}$ von Neumann algebra with
separable predual. Suppose $\rho$ is a faithful normal tracial state
on $\mathcal{M}$. Then by Lemma \ref{3.2},  the $2$-norm induced by
$\rho$ gives the same topology as the strong operator topology on
bounded subsets of $\mathcal{M}$. The unit ball $(\mathcal{M})_{1}$ is
a metric space under this $2$-norm. Using a similar argument as Section
4 in \cite{EFAR2}, we can get the joint continuity of $\phi$ on
$(\mathcal{M})_{1} \times (\mathcal{M})_{1} \times \dots \times
(\mathcal{M})_{1}$ in the $2$-norm induced by $\rho$. Therefore we have
the following lemma.

\begin{lemma} \label{5.2}
Let $\mathcal{M}$ be a type II$_{1}$ von Neumann algebra with
separable predual and $\phi : \mathcal{M}^{k} \to B(H)$ be a bounded
$k$-linear separately normal map. Let $\rho$ be a faithful normal
tracial state on $\mathcal{M}$. Suppose $\{ p_{t} : t \in \mathbb{N}
\}$ is a sequence of projections in $\mathcal{M}$ satisfying (\ref{equa 5.1.1}), (\ref{equa 5.1.2}) and
$$\lim\limits_{t \to \infty} \Vert p_{t}a-ap_{t} \Vert_{2} = 0$$
for any $a \in \mathcal{M}$, where $\Vert \cdot \Vert_{2}$ is the $2$-norm induced by $\rho$. Then for any $a_{1}, a_{2}, \dots, a_{k} \in \mathcal{M}$, each $\sigma \in S_{k}$, each integer $i \leq l(\sigma)$ and each pair of unit vectors $x, y \in H$,
$$\lim\limits_{t \to \infty} \langle \phi_{\sigma, p_{t}, i} (a_{1}, a_{2}, \dots, a_{k})x, y \rangle =0,$$
and $$\lim\limits_{t \to \infty} \langle p_t\phi_{\sigma,
p_t}(a_{1}, \dots, a_{k})x, y \rangle =0.$$
\end{lemma}
\begin{proof}The proof is similar to the one of Lemma 6.2 in \cite{EFAR2} and is skipped here.
\end{proof}

Now we have the following result.
\begin{theorem} \label{5.3}
Let $\mathcal{M}$ be a type II$_{1}$ von Neumann algebra with
separable predual and   $\mathcal{Z}$ be the center of
$\mathcal{M}$. Let $\rho$ be a faithful normal tracial state on
$\mathcal{M}$ and $\| \cdot \|_2$ be the $2$-norm induced by $\rho$
on $\mathcal M$. Suppose $\mathcal{R}$ is a hyperfinite type
II$_{1}$ subfactor of $\mathcal{M}$ such that
\begin{enumerate}
\item [(I)] $\mathcal{R}' \cap \mathcal{M} = \mathcal{Z}$;\\
\item [(II)] for any $n \in \mathbb{N}$, any elements $a_{1}, a_{2}, ..., a_{m}$ in $\mathcal{M}$, there exists a countable collection of projections $\{ p_{1t}, p_{2t}, \dots, p_{nt} : t \in \mathbb{N} \}$ in $\mathcal{R}$ such that
\begin{enumerate}
\item [(i)] for each $t \in \mathbb{N}$, $p_{1t}, p_{2t}, \dots, p_{nt}$ are mutually orthogonal equivalent projections in $\mathcal{R}$ with sum I, the identity in $\mathcal{M}$;

\item [(ii)] $\lim\limits_{t \to \infty}\left\Vert p_{it}a_{l}-a_{l}p_{it} \right\Vert_{2} = 0$ for any $i=1, 2,\dots , n,  l=1, 2, \dots, m$.
\end{enumerate}
\end{enumerate}
Then a bounded $k$-linear $\mathcal{R}$-multimodular separately normal map $\phi : \mathcal{M}^{k} \to B(H)$ is completely bounded and $\left\Vert \phi \right\Vert_{cb} \leq 2^{k} \left\Vert \phi \right\Vert$.
\end{theorem}

\begin{proof}
The proof is similar to the one for  Theorem 6.3 in \cite{EFAR2} and is
sketched here for the purpose of completeness.

Fix $n \in \mathbb{N}$ and $k$ elements $b_{1}, b_{2}, \dots, b_{k} \in M_{n}(\mathcal{M})$.

By condition (II), we can find a family of projections $\{ q_{it} : 1 \leq i \leq n; t \in \mathbb{N}\}$ in $\mathcal{R}$ such that
\begin{enumerate}
\item [(a)] for each $t \in \mathbb{N}$, $q_{1t}, \dots, q_{nt}$ are $n$ orthogonal equivalent projections in $\mathcal{R}$ with sum $I$;

\item [(b)] $\lim\limits_{t \to \infty}\left\Vert q_{it}a-aq_{it} \right\Vert_{2} = 0$ for any $a \in \mathcal{M}, 1 \leq i \leq n$.
\end{enumerate}
\noindent Let $q'_{it}= I_{n} \otimes q_{it} \in M_{n}(\mathcal{R})$ for each $i$ and $t$. We obtain that
\begin{enumerate}
\item [(a')] for each $t \in \mathbb{N}$, $q'_{1t}, \dots, q'_{nt}$ are $n$ orthogonal equivalent projections in $M_{n}(\mathcal{R})$ with sum $I_{n} \otimes I$;

\item [(b')] $\lim\limits_{t \to \infty}\left\Vert q'_{it}b-bq'_{it} \right\Vert_{2} = 0$ for any $b \in M_{n}(\mathcal{M}), 1 \leq i \leq n$.
\end{enumerate}
Since $\phi$ is an $\mathcal{R}$-multimodular map, $\phi^{(n)}$ is
an $M_{n}(\mathcal{R})$-multimodular map. Assume that $\mathcal M$
acts on a Hilbert space $H$.  For any two unit vectors $x, y$ in
$H^{n}$ and any $t \in \mathbb{N}$, by Lemma \ref{5.1},
\begin{eqnarray*}
&&\langle \phi^{(n)}(b_{1}, \dots, b_{k})x, y \rangle\\
&=& \langle \sum\limits_{i=1}^{n} q'_{it} \phi^{(n)} (b_{1}, \dots, b_{k})x, y \rangle\\
&=&\langle \sum\limits_{i=1}^{n} \sum\limits_{\sigma \in S_{k}} (-1)^{\vert \sigma \vert +1}
q'_{it} \phi^{(n)}_{\sigma, q'_{it}}(b_{1}, \dots, b_{k})x, y \rangle+\langle \sum\limits_{i=1}^{n}q'_{it} \phi^{(n)}(b_{1}q'_{it}, \dots, b_{k}q'_{it})x, y \rangle\\
&=&\langle \sum\limits_{i=1}^{n} \sum\limits_{\sigma \in S_{k}} (-1)^{\vert \sigma \vert +1} q'_{it} \phi^{(n)}_{\sigma, q'_{it}}(b_{1}, \dots, b_{k})x, y \rangle+\langle \sum\limits_{i=1}^{n}q'_{it} \phi^{(n)}(b_{1}q'_{it}, \dots, b_{k}q'_{it})q'_{it}x, y \rangle.
\end{eqnarray*}

Therefore
\begin{eqnarray}
&&\langle \phi^{(n)}(b_{1}, \dots, b_{k})x, y \rangle - \langle \sum\limits_{i=1}^{n} \sum\limits_{\sigma \in S_{k}} (-1)^{\vert \sigma \vert +1} q'_{it} \phi^{(n)}_{\sigma, q'_{it}}(b_{1}, \dots, b_{k})x, y \rangle \notag \\
&=&\langle \sum\limits_{i=1}^{n}q'_{it} \phi^{(n)}(b_{1}q'_{it}, \dots, b_{k}q'_{it})q'_{it}x, y \rangle. \label{74}
\end{eqnarray}

Since $\{ q'_{1t}, \dots, q'_{nt} \}$ is a set of $n$ orthogonal projections for each $t \in \mathbb{N}$, by Corollary \ref{4.3},
\begin{eqnarray}
\Vert \sum\limits_{i=1}^{n} q'_{it} \phi^{(n)}(b_{1}q'_{it}, \dots, b_{k}q'_{it})q'_{it} \Vert &\leq& \max\limits_{1 \leq i \leq n} \Vert q'_{it} \phi^{(n)}(b_{1}q'_{it}, \dots, b_{k}q'_{it})q'_{it} \Vert \notag \\
&\leq& 2^{k} \Vert \phi \Vert \Vert b_{1} \Vert \dots \left\Vert b_{k} \right\Vert. \label{76}
\end{eqnarray}

By Lemma \ref{5.2}, condition (b') implies
\begin{eqnarray}
\lim\limits_{t \to \infty} \langle q'_{it} \phi^{(n)}_{\sigma, q'_{it}}(b_{1}, \dots, b_{k})x, y \rangle = 0 \label{77}
\end{eqnarray}
for each $1 \leq i \leq n$ and $\sigma \in S_{k}$.

Letting $t \to \infty$ for both sides of (\ref{74}), it follows from inequality (\ref{76}) and equation (\ref{77}) that
$$\langle \phi^{(n)}(b_{1}, \dots, b_{k})x, y \rangle \leq 2^{k} \Vert \phi \Vert \Vert b_{1} \Vert \dots \Vert b_{k} \Vert.$$
Since $n, x, y$ were arbitrarily chosen, $\Vert \phi \Vert_{cb} \leq 2^{k} \Vert \phi \Vert$.
\end{proof}

The following is the main result of the paper.

\begin{theorem}\label{mainthm}
If $\mathcal{M}$ is a type II$_{1}$ von Neumann algebra with
separable predual and Property $\Gamma$, then
the Hochschild cohomology group
$$H^{k}(\mathcal{M}, \mathcal{M})=0, \qquad \forall \ k \geq 2.$$
\end{theorem}

\begin{proof}
By Theorem \ref{3.13}, there is a hyperfinite type II$_{1}$
subfactor  $\mathcal{R}$ of $\mathcal{M}$ satisfying conditions (I)
and (II) in Theorem \ref{5.3}.

Now consider the cohomology groups $H^{k}(\mathcal{M},
\mathcal{M})$. By Theorem 3.1.1 in \cite{AR3}, it suffices to
consider a $k$-linear $\mathcal{R}$-multimodular separately normal
cocycle $\phi$. Theorem 5.3 shows that such cocycles are completely
bounded. By Theorem 4.3.1 in \cite{AR3}, completely bounded
Hochschild cohomology groups are trivial. It follows that $\phi$ is
a coboundary, whence $H^{k}(\mathcal{M}, \mathcal{M})=0, k \geq 2$.
\end{proof}

The next result in \cite{EGA} follows directly from Theorem
\ref{mainthm} and Example \ref{example 1}.
\begin{corollary}\label{mainthm2}
Suppose that $\mathcal{M}_1$ is a type II$_{1}$ von Neumann algebra
with separable predual and $\mathcal{M}_2$ is a type II$_{1}$ factor
with separable predual. If $\mathcal M_2$ has Property $\Gamma$,
then  the Hochschild cohomology group
$$H^{k}(\mathcal{M}_1\otimes \mathcal{M}_2, \mathcal{M}_1\otimes \mathcal{M}_2)=0,  \qquad  k \geq
2.$$ In particular, if $\mathcal M$ is a type II$_{1}$ von Neumann
algebra with separable predual satisfying $\mathcal M \cong \mathcal
M\otimes \mathcal R$, where $\mathcal R$ is the hyperfinite II$_1$
factor, then $$H^{k}(\mathcal{M}, \mathcal{M})=0, \qquad  \forall \ k \geq
2.$$
\end{corollary}

\end{document}